\documentclass[12pt]{article} 
\usepackage{amsfonts,amsmath,latexsym,amssymb,mathrsfs,amsthm}

\def\numberlikeadb{\global\def\theequation{\thesection.\arabic{equation}}}
\numberlikeadb
\newtheorem{theorem}{Theorem}[section]
\newtheorem{lemma}[theorem]{Lemma}
\newtheorem{corollary}[theorem]{Corollary}
\newtheorem{definition}[theorem]{Definition}
\newtheorem{proposition}[theorem]{Proposition}
\newtheorem{remark}[theorem]{Remark}

\numberwithin{equation}{section}

\begin{document}

\title{Variance-Gamma approximation via Stein's method}
\author{Robert E. Gaunt\footnote{Department of Statistics,
University of Oxford, 1 South Parks Road, OXFORD OX1 3TG, UK; supported by an EPSRC research grant. }
}

\date{March 2014} 
\maketitle

\begin{abstract}Variance-Gamma distributions are widely used in financial modelling and contain as special cases the normal, Gamma and Laplace distributions.  In this paper we extend Stein's method to this class of distributions.  In particular, we obtain a Stein equation and smoothness estimates for its solution.  This Stein equation has the attractive property of reducing to the known normal and Gamma Stein equations for certain parameter values.  We apply these results and local couplings to bound the distance between sums of the form $\sum_{i,j,k=1}^{m,n,r}X_{ik}Y_{jk}$, where the $X_{ik}$ and $Y_{jk}$ are independent and identically distributed random variables with zero mean, by their limiting Variance-Gamma distribution.  Through the use of novel symmetry arguments, we obtain a bound on the distance that is of order $m^{-1}+n^{-1}$ for smooth test functions.  We end with a simple application to binary sequence comparison. 
\end{abstract}

\noindent{{\bf{Keywords:}}} Stein's method, Variance-Gamma approximation, rates of convergence

\noindent{{{\bf{AMS 2010 Subject Classification:}}} 60F05

\section{Introduction}In 1972, Stein \cite{stein} introduced a powerful method for deriving bounds for normal approximation.  Since then, this method has been extended to many other distributions, such as the Poisson \cite{chen 0}, Gamma \cite{luk}, \cite{nourdin1}, Exponential \cite{chatterjee}, \cite{pekoz1} and Laplace \cite{dobler}, \cite{pike}.  Through the use of differential or difference equations, and various coupling techniques, Stein's method enables many types of dependence structures to be treated, and also gives explicit bounds for distances between distributions.

At the heart of Stein's method lies a characterisation of the target distribution and a corresponding characterising differential or difference equation.  For example, Stein's method for normal approximation rests on the following characterization of the normal distribution, which can be found in Stein \cite{stein2}, namely $Z\sim N(\mu,\sigma^2)$ if and only if 
\begin{equation} \label{stein lemma}\mathbb{E}[\sigma^2f'(Z)-(Z-\mu)f(Z)]=0
\end{equation}
for all sufficiently smooth $f$.  This gives rise to the following inhomogeneous differential equation, known as the Stein equation:
\begin{equation} \label{normal equation} \sigma^2f'(x)-(x-\mu)f(x)=h(x)-\mathbb{E}h(Z),
\end{equation}
where $Z \sim N(\mu,\sigma^2)$, and the test function $h$ is a real-valued function.  For any bounded test function, a solution $f$ to (\ref{normal equation}) exists (see Lemma 2.4 of Chen et al$.$ \cite{chen}).  There are a number of techniques for obtaining Stein equations, such as the density approach of Stein et al$.$ \cite{stein3}, the scope of which  has recently been extended by Ley and Swan \cite{ley}.  Another commonly used technique is a generator approach, introduced by Barbour \cite{barbour2}.  This approach involves recognising the target the distribution as the stationary distribution of Markov process and then using the theory of generators of stochastic process to arrive at a Stein equation; for a detailed overview of this method see Reinert \cite{reinert 0}.  Luk \cite{luk} used this approach to obtain the following Stein equation for the $\Gamma(r,\lambda)$ distribution:
\begin{equation} \label{delight} xf''(x)+(r-\lambda x)f'(x)=h(x)-\mathbb{E}h(X), 
\end{equation}
where $X\sim \Gamma(r,\lambda)$.  

The next essential ingredient of Stein's method is smoothness estimates for the solution of the Stein equation.  This can often be done by solving the Stein equation using standard solution methods for differential equations and then using direct calculations to bound the required derivatives of the solution (Stein \cite{stein2} used the approach to bound the first two derivatives of the solution to the normal Stein equation (\ref{normal equation})).  The generator approach is also often used to obtain smoothness estimates.  The use of probabilistic arguments to bound the derivatives of the solution often make it easier to arrive at smoothness estimates than through the use of analytical techniques.  Luk \cite{luk} and Pickett \cite{pickett} used the generator approach to bound $k$-th order derivatives of the solution of the $\Gamma(r,\lambda)$ Stein equation (\ref{delight}). Pickett's bounds are as follows
\begin{equation}\label{pickbound1}\qquad \|f^{(k)}\|\leq\bigg\{\sqrt{\frac{2\pi}{r}}+\frac{2}{r}\bigg\}\|h^{(k-1)}\|, \qquad k\geq 1,
\end{equation}
where $\|f\|=\|f\|_{\infty}=\sup_{x\in\mathbb{R}}|f(x)|$ and $h^{(0)}\equiv h$.

In this paper we obtain the key ingredients required to extend Stein's method to the class of Variance-Gamma distributions.  The Variance-Gamma distributions are defined as follows (this parametrisation is similar to that given in Finlay and Seneta \cite{finlay}).
\begin{definition}[$\mathbf{Variance\mbox{\boldmath{-}}Gamma}$ $\mathbf{distribution,}$ $\mathbf{first}$ $\mathbf{parametrisation}$]
The random variable $X$ is said to have a \emph{Variance-Gamma} distribution with parameters $r > 0$, $\theta \in \mathbb{R}$, $\sigma >0$, $\mu \in \mathbb{R}$ if and only if it has probability density function given by
\begin{equation} \label{vgdef} p_{\mathrm{VG}_1}(x;r,\theta,\sigma,\mu) = \frac{1}{\sigma\sqrt{\pi} \Gamma(\frac{r}{2})} e^{\frac{\theta}{\sigma^2} (x-\mu)} \bigg(\frac{|x-\mu|}{2\sqrt{\theta^2 +  \sigma^2}}\bigg)^{\frac{r-1}{2}} K_{\frac{r-1}{2}}\bigg(\frac{\sqrt{\theta^2 + \sigma^2}}{\sigma^2} |x-\mu| \bigg), 
\end{equation}
where $x \in \mathbb{R}$, and $K_{\nu}(x)$ is a modified Bessel function of the second kind; see Appendix B for a definition.  If (\ref{vgdef}) holds then we write $X\sim \mathrm{VG}_1(r,\theta,\sigma,\mu)$.
\end{definition}

The density (\ref{vgdef}) may at first appear to be undefined in the limit $\sigma \rightarrow 0$, but this limit does in fact exist and this can easily be verified from the asymptotic properties of the modified Bessel function $K_{\nu}(x)$ (see formula (\ref{Ktendinfinity}) from Appendix B).  As we shall see in Proposition \ref{norgamlap} (below), taking the limit $\sigma\rightarrow 0$ and putting $\mu=0$ gives the family of Gamma distributions.  It is also worth noting that the support of the Variance-Gamma distributions is $\mathbb{R}$ when $\sigma>0$, but in the limit $\sigma\rightarrow 0$ the support is the region $(\mu,\infty)$ if $\theta>0$, and is $(-\infty,\mu)$ if $\theta<0$.

The Variance-Gamma distributions were introduced to the financial literature by Madan and Seneta \cite{madan}.  For certain parameter values the Variance-Gamma distributions have semi heavy tails that decay slower than the tails of the normal distribution, and therefore are often appropriate for financial modelling.  

The class of Variance-Gamma distributions includes the Laplace distribution as a special case and in the appropriate limits reduces to the normal and Gamma distributions.  This family of distributions also contains many other distributions that are of interest, which we list in the following proposition (the proof is given in Appendix A).  As far as the author is aware, this is the first list of characterisations of the Variance-Gamma distributions to appear in the literature.

\begin{proposition}\label{norgamlap}(i) Let $\sigma>0$ and $\mu\in\mathbb{R}$ and suppose that $Z_r$ has the $\mathrm{VG}_1(r,0,\sigma/\sqrt{r},\mu)$ distribution.  Then $Z_r$ converges in distribution to a $N(\mu,\sigma^2)$ random variable in the limit $r\rightarrow\infty$.

(ii) Let $\sigma>0$ and $\mu\in\mathbb{R}$, then a $\mathrm{VG}_1(2,0,\sigma,\mu)$ random variable has the $\mathrm{Laplace}(\mu,\sigma)$ distribution with probability density function 
\begin{equation}\label{laplacepdf}p_{\mathrm{VG_1}}(x;2,0,\sigma,\mu) =\frac{1}{2\sigma} \exp\bigg(-\frac{|x-\mu|}{\sigma}\bigg), \quad x \in \mathbb{R}.
\end{equation}

(iii) Suppose that $(X,Y)$ has the bivariate normal distribution with correlation $\rho$ and marginals $X\sim N(0,\sigma_X^2)$ and $Y\sim N(0,\sigma_Y^2)$. Then the product $XY$ follows the $\mathrm{VG}_1(1,\rho\sigma_X\sigma_Y, \sigma_X\sigma_Y\sqrt{1-\rho^2},0)$ distribution.

(iv) Let $X_1,\ldots,X_r$ and $Y_1,\ldots,Y_r$ be independent standard normal random variables.  Then $\mu+\sigma\sum_{k=1}^rX_kY_k$ has the $\mathrm{VG}_1(r,0,\sigma,\mu)$ distribution.  As a special case we have that a Laplace random variable with density (\ref{laplacepdf}) has the representation $\mu+\sigma(X_1Y_1+X_2Y_2)$.

(v) The Gamma distribution is a limiting case of the Variance-Gamma distributions: for $r>0$ and $\lambda>0$, the random variable $X_{\sigma}\sim \mathrm{VG}_1(2r,(2\lambda)^{-1},\sigma,0)$ convergences in distribution to a $\Gamma(r,\lambda)$ random variable in the limit $\sigma\downarrow 0$.

(vi) Suppose that $(X,Y)$ follows a bivariate gamma distribution with correlation $\rho$ and marginals $X\sim \Gamma(r,\lambda_1)$ and $Y\sim \Gamma(r,\lambda_2)$.  Then the random variable $X-Y$ has the $\mathrm{VG}_1(2r,(2\lambda_1)^{-1}-(2\lambda_2)^{-1},(\lambda_1\lambda_2)^{-1/2}(1-\rho)^{1/2},0)$ distribution.

\end{proposition}

The representations of the Variance-Gamma distributions given in Proposition \ref{norgamlap} enable us to determine a number of statistics that may have asymptotic Variance-Gamma distributions.  

One of the main results of this paper (see Lemma \ref{ghds}) is the following Stein equation for the Variance-Gamma distributions:  
\begin{equation} \label{nice} \sigma^2 (x-\mu)f''(x) + (\sigma^2 r + 2\theta (x-\mu))f'(x) +(r\theta -  (x-\mu))f(x) = h(x) -\mathrm{VG}^{r,\theta}_{\sigma,\mu}h,
\end{equation}
where $\mathrm{VG}^{r,\theta,}_{\sigma,\mu}h$ denotes the quantity $\mathbb{E}h(X)$ for $X \sim \mathrm{VG}_1(r,\theta,\sigma,\mu)$.  We also  obtain uniform bounds for the first four derivatives of the solution of the Stein equation for the case $\theta=0$.  

In Section 3, we analyse the Stein equation (\ref{nice}).  In particular, we show that the normal Stein equation (\ref{normal equation}) and Gamma Stein equation (\ref{delight}) are special cases.   As a Stein equation for a given distribution is not unique (see Barbour \cite{barbour1}), the fact that in the appropriate limit the Variance-Gamma Stein equation (\ref{nice}) reduces to the known normal and Gamma Stein equation is an attractive feature.  

Stein's method has also recently been extended to the Laplace distribution (see Pike and Ren \cite{pike} and D\"obler \cite{dobler}), although the Laplace Stein equation obtained by \cite{pike} differs from the Laplace Stein equation that arises as a special case of (\ref{nice}); see Section 3.1.1 for a more detailed discussion.  Another special case of the Stein equation (\ref{nice}) is a Stein equation for the product of two independent central normal random variables, which is in agreement with the Stein equation for products of independent central normal that was recently obtained by Gaunt \cite{gaunt pn}.   Therefore, the results from this paper allow the existing literature for Stein's method for normal, Gamma, Laplace and product normal approximation to be considered in a more general framework.

More importantly, our development of Stein's method for the Variance-Gamma distributions allows a number of new situations to be treated by Stein's method.  In Section 4, we illustrate our method by obtaining a bound for the distance between the statistic 
\begin{equation}\label{wreqn}W_r=\sum_{i,j,k=1}^{m,n,r}X_{ik}Y_{jk}=\sum_{k=1}^r\bigg(\frac{1}{\sqrt{m}}\sum_{i=1}^mX_{ik}\bigg)\bigg(\frac{1}{\sqrt{n}}\sum_{j=1}^nY_{jk}\bigg),
\end{equation}
where the $X_{ik}$ and $Y_{jk}$ are independent and identically distributed with zero mean, and its asymptotic distribution, which, by the central limit theorem and part (iv) of Proposition \ref{norgamlap}, is the $\mathrm{VG}_1(r,0,1,0)$ distribution.  By using the $\mathrm{VG}_1(r,0,1,0)$ Stein equation
\begin{equation}\label{nicerr}xf''(x)+rf'(x)-xf(x)=h(x)-\mathrm{VG}^{r,0}_{1,0}h,
\end{equation} 
local approach couplings, and symmetry arguments, that were introduced by Pickett \cite{pickett}, we obtain a $O(m^{-1}+n^{-1})$ bound for smooth test functions.  A similar phenomena was observed in chi-square approximation by Pickett, and also by Goldstein and Reinert \cite{goldstein} in which they obtained $O(n^{-1})$ convergence rates in normal approximation, for smooth test functions, under the assumption of vanishing third moments.  For non-smooth test functions we would, however, expect a $O(m^{-1/2}+n^{-1/2})$ convergence rate (cf. Berry-Ess\'{e}en Theorem (Berry \cite{berry} and Ess\'{e}en \cite{esseen}) to hold; see Remark \ref{undone}.  

The rest of this paper is organised as follows.  In Section 2, we introduce the Variance-Gamma distributions and state some of their standard properties.    In Section 3, we obtain a characterising lemma for the Variance-Gamma distributions and a corresponding Stein equation.  We also obtain the unique bounded solution of the Stein equation, and present uniform bounds for the first four derivatives of the solution for the case $\theta =0$.  In Section 4, we use Stein's method for Variance-Gamma approximation to bound the distance between the statistic (\ref{wreqn}) and its limiting Variance-Gamma distribution.  We then apply this bound to an application of binary sequence comparison, which is a simple special case of the more general problem of word sequence comparison.  In Appendix A, we include the proofs of some technical lemmas that are required in this paper.  Appendix B provides a list of some elementary properties of modified Bessel functions that we make use of in this paper. 

\section{The class of Variance-Gamma distributions}

In this section we present the Variance-Gamma distributions and some of their standard properties.  Throughout this paper we will make use of two different parametrisations of the Variance-Gamma distributions; the first parametrisation was given in Section 1, and making the change of variables
\begin{equation} \label{parameter} \nu=\frac{r-1}{2}, \qquad \alpha =\frac{\sqrt{\theta^2 +  \sigma^2}}{\sigma^2}, \qquad \beta =\frac{\theta}{\sigma^2}.
\end{equation}
leads to another useful parametrisation.  This parametrisation can be found in Eberlein and Hammerstein \cite{eberlein}.
\begin{definition}[$\mathbf{Variance\mbox{\boldmath{-}}Gamma}$ $\mathbf{distribution}$, $\mathbf{second}$ $\mathbf{parametrisation}$]
The random variable $X$ is said to have a \emph{Variance-Gamma} distribution with parameters $\nu,\alpha,\beta,\mu$, where $\nu > -1/2$, $\mu \in \mathbb{R}$, $\alpha >|\beta|$, if and only if its probability density function is given by
\begin{equation} \label{seven} p_{\mathrm{VG}_2}(x;\nu,\alpha,\beta,\mu) = \frac{(\alpha^2 - \beta^2)^{\nu + 1/2}}{\sqrt{\pi} \Gamma(\nu + \frac{1}{2})}\left(\frac{|x-\mu|}{2\alpha}\right)^{\nu} e^{\beta (x-\mu)}K_{\nu}(\alpha|x-\mu|), \quad x\in \mathbb{R}.
\end{equation}
If (\ref{seven}) holds then we write $X\sim \mathrm{VG}_2(\nu,\alpha,\beta,\mu)$.
\end{definition}

\begin{definition}If $X\sim \mathrm{VG}_1(r,0,\sigma,\mu)$, for $r$, $\sigma$, and $\mu$ defined as in Definition \ref{vgdef} (or equivalently $X\sim \mathrm{VG}_2(\nu,\alpha,0,\mu)$), then $X$ is said to have a \emph{Symmetric Variance-Gamma} distribution.
\end{definition}

The first parametrisation leads to simple characterisations of the Variance-Gamma distributions in terms of normal and Gamma distributions, and therefore in many cases it allows us to recognise statistics that will have an asymptotic Variance-Gamma distribution.  For this reason, we state our main results in terms of this parametrisation.  However, the second parametrisation proves to be very useful in simplifying the calculations of Section 3, as the solution of the Variance-Gamma Stein equation has a simpler representation for this parametrisation.  We can then state the results in terms of the first parametrisation by using (\ref{parameter}).

The Variance-Gamma distributions have moments of arbitrary order (see Eberlein and Hammerstein \cite{eberlein}), in particular the mean and variance (for both parametrisations) of a random variable $X$ with a Variance-Gamma distribution are given by
\begin{eqnarray} \label{vgmoment}\mathbb{E}X&=&\mu+\frac{ (2\nu+1)\beta}{\alpha^2-\beta^2}=\mu+ r\theta, \\
\mathrm{Var}X&=&\frac{2\nu+1}{\alpha^2-\beta^2}\left(1+\frac{2\beta^2}{\alpha^2-\beta^2}\right)=r( \sigma^2+2\theta^2).\nonumber
\end{eqnarray}

The following proposition, which can be found in Bibby and S{\o}rensen \cite{bibby}, shows that the class of Variance-Gamma distributions is closed under convolution, provided that the random variables have common values of $\theta$ and $\sigma$ (or, equivalently, common values of $\alpha$ and $\beta$  in the second parametrisation).  
\begin{proposition} \label{tour}Let $X_1$ and $X_2$ be independent random variables such that $X_i \sim \mathrm{VG}_1(r_i ,\theta,\sigma,\mu_i)$, $i=1,2$, then we have that
\[X_1 +X_2 \sim \mathrm{VG}_1(r_1+r_2 ,\theta,\sigma,\mu_1+\mu_2).\]
\end{proposition}
Variance-Gamma random variables can be characterised in terms of independent normal and Gamma random variables.  This characterisation is given in the following proposition, which can be found in Barndorff-Nielsen et al$.$ \cite{barn}. 

\begin{proposition} \label{giro} Let $r > 0$, $\theta \in \mathbb{R}$, $\sigma > 0$ and  $\mu \in \mathbb{R}$.  Suppose that $U\sim N(0,1)$ and $V\sim \Gamma(r/2,1/2)$ are independent random variables and let $Z \sim \mathrm{VG}_1(r,\theta,\sigma,\mu)$, then
\begin{equation*}\label{vgcharvg}Z \stackrel{\mathcal{D}}{=} \mu + \theta V +\sigma \sqrt{V} U.
\end{equation*}
\end{proposition}
Using Proposition \ref{giro} we can establish the following useful representation of the Variance-Gamma distributions, which appears to be a new result.  Indeed, the representation allows us to see that the statistic (\ref{wreqn}) has an asymptotic Variance-Gamma distribution. 
\begin{corollary} \label{vuelta}
Let $\theta \in \mathbb{R}$, $\sigma > 0$, $\mu \in \mathbb{R}$, and $r$ be a positive integer.  Let $X_1,X_2,\ldots,X_r$ and $Y_1,Y_2,\ldots,Y_r$ be independent standard normal random variables and let $Z$ be a $\mathrm{VG}_1(r,\theta,\sigma,\mu)$ random variable, then

\begin{equation*}Z \stackrel{\mathcal{D}}{=} \mu + \theta \sum_{i=1}^r X_i^2 + \sigma \sum_{i=1}^r X_i Y_i. \end{equation*}
\end{corollary}  

\begin{proof}Let $X_1,X_2,...,X_r$ and $Y_1,Y_2,...,Y_r$ be sequences of independent standard normal random variables.  Then $X_i^2$, $i=1,2,...,m$, has a $\chi_{(1)}^2$ distribution, that is a $\Gamma(1/2,1/2)$ distribution.  Define
\[Z_1=\mu+\theta X_1^2 +\sigma X_1Y_1, \qquad
Z_i = \theta X_i^2 +\sigma X_iY_i, \quad i=2,3,\ldots ,r.
\]
Note that $X_iY_i\stackrel{\mathcal{D}}{=}|X_i|Y_i$.  Hence, by Proposition \ref{giro}, we have that $Z_1$ is a $\mathrm{VG}_1(1,\theta,\sigma,\mu)$ random variable and $Z_i$, $i=2,\ldots r$, are $\mathrm{VG}_1(1,\theta,\sigma,0)$ random variables.  It therefore follows from Proposition \ref{tour} that the sum $Z=\sum_{i=1}^rZ_i$ follows the $\mathrm{VG}_1(r,\theta,\sigma,\mu)$ distribution.  
\end{proof}  

\section{Stein's method for Variance-Gamma distributions}

\subsection{A Stein equation for the Variance-Gamma distributions}

The following lemma, which characterises the Variance-Gamma distributions, will lead to a Stein equation for the Variance-Gamma distributions.  Before stating the lemma, we note that an application of the asymptotic formula (\ref{Ktendinfinity}) to the density function (\ref{seven}) allows us to deduce the tail behaviour of the $\mathrm{VG}_2(\nu,\alpha,\beta,\mu)$ distribution: 
\begin{equation}\label{ptail}p_{\mathrm{VG}_2}(x;\nu,\alpha,\beta,\mu)\sim\begin{cases} \displaystyle \frac{1}{\pi\Gamma(\nu+\frac{1}{2})}\bigg(\frac{\alpha^2-\beta^2}{2\alpha}\bigg)^{\nu+\frac{1}{2}}x^{\nu-\frac{1}{2}}\mathrm{e}^{-(\alpha-\beta)(x-\mu)}, & \:  x\rightarrow\infty, \\[10pt]
\displaystyle \frac{1}{\pi\Gamma(\nu+\frac{1}{2})}\bigg(\frac{\alpha^2-\beta^2}{2\alpha}\bigg)^{\nu+\frac{1}{2}}(-x)^{\nu-\frac{1}{2}}\mathrm{e}^{(\alpha+\beta)(x-\mu)}, & \: x\rightarrow-\infty.
\end{cases}
\end{equation}
Note that the tails are in general not symmetric.
  
\begin{lemma} \label{ghds}
Let $W$ be a real-valued random variable.  Then $W$ follows the $\mathrm{VG}_2(\nu,\alpha,\beta,\mu)$ distribution if and only if
\begin{equation} \label{lemmaone}\mathbb{E}[(W-\mu)f''(W) + (2\nu+1 +2\beta (W-\mu))f'(W) + ((2\nu+1)\beta - (\alpha^2 -\beta^2)(W-\mu))f(W)] = 0
\end{equation}
for all piecewise twice continuously differentiable functions $f:\mathbb{R}\rightarrow\mathbb{R}$ that satisfy
\begin{equation}\label{limf1}\lim_{x\rightarrow\infty}f^{(k)}(x)x^{\nu+3/2}\mathrm{e}^{-(\alpha-\beta)x}=0 \qquad \text{and} \qquad  \lim_{x\rightarrow-\infty}f^{(k)}(x)(-x)^{\nu+3/2}\mathrm{e}^{(\alpha+\beta)x}=0
\end{equation}
for $k=0,1,2$, where $f^{(0)}\equiv f$.
\end{lemma} 

\begin{proof}To simplify the calculations, we prove the result for the special case $\mu=0$, $\alpha =1$, $-1<\beta<1$.  For $W=\alpha (Z-\mu)$ we have that $W \sim \mathrm{VG}_2(\nu,1,\beta,0)$ if and only if $Z \sim \mathrm{VG}_2(\nu,\alpha,\alpha\beta,\mu)$, and so we can deduce the general case by applying a simple linear transformation. 

\emph{Necessity}.  Suppose that $W \sim \mathrm{VG}_2(\nu,1,\beta,0)$.  We split the range of integration to obtain
\[ \label{mint} \mathbb{E}[Wf''(W)+(2\nu +1+2\beta W)f'(W)+((2\nu +1)\beta -(1-\beta^2)W)f(W)] = I_1 +I_2,
\]
where
\begin{align*}I_1&=\int_{0}^\infty\{xf''(x)+(2\nu +1+2\beta x)f'(x)+((2\nu +1)\beta -(1-\beta^2)x)f(x)\} p(x)\,\mathrm{d}x, \\
I_2&=\int_{-\infty}^{0}\{xf''(x)+(2\nu +1+2\beta x)f'(x)+((2\nu +1)\beta -(1-\beta^2)x)f(x)\}p(x)\,\mathrm{d}x,
\end{align*}
and $p(x)=\kappa_{\nu,\beta}x^{\nu}\mathrm{e}^{\beta x}K_{\nu}(x)$, where $\kappa_{\nu,\beta}$ is the normalising constant, is the density of $W$.  The integrals $I_1$ and $I_2$ exist because $f$ is piecewise twice continuously differentiable that satisfies the conditions of (\ref{limf1}), which on recalling the tail behaviour of $p(x)$ given in (\ref{ptail}) ensures that, for $k=0,1,2$, we have that $xp(x)f^{(k)}(x)=o(|x|^{-1})$ as $|x|\rightarrow\infty$.  

Firstly, we consider $I_1$.  Let $A(x)=x$, $B(x)=2\nu +1+2\beta x$ and $C(x)=(2\nu +1)\beta -(1-\beta^2)x$.  Then applying integration by parts twice gives
\begin{align*}I_1&=\int_0^\infty \{A(x)p''(x)+(2A'(x)-B(x))p'(x)+(A''(x)-B'(x)+C(x))p(x)\}f(x)\,\mathrm{d}x\\
&\quad+\Big[A(x)p(x)f'(x)\Big]_0^\infty  +\Big[\{B(x)p(x)-(A(x)p(x))'\}f(x)\Big]_0^\infty
\end{align*}
\begin{align*}
&=\int_0^{\infty}\{xp''(x)+(-2\nu+1-2\beta x)p'(x)+((2\nu-1)\beta-(1-\beta^2)x)p(x)\}f(x)\,\mathrm{d}x\\
&\quad+\Big[xp(x)f'(x)\Big]_0^\infty  +\Big[(2\nu +2\beta x)p(x)f(x)-xp'(x)f(x)\Big]_0^\infty .
\end{align*}

Straightforward differentiation of the function $p(x)=\kappa_{\nu,\beta}x^{\nu}\mathrm{e}^{\beta x}K_{\nu}(x)$ shows that the integrand in the above display is equal to
\begin{align*}\kappa_{\nu,\beta}x^{\nu-1}\mathrm{e}^{\beta x}\{x^2K_{\nu}''(x)+xK_{\nu}'(x)-(x^2+\nu^2)K_{\nu}(x)\}f(x)=0,
\end{align*}
as $K_{\nu}(x)$ is a solution of the modified Bessel differential equation (see (\ref{realfeel})).  

We now note that $p(x)=O(x^{\nu-1/2}\mathrm{e}^{-(1-\beta)x})$ as $x\rightarrow\infty$ (see (\ref{ptail})), and by differentiating $p(x)$ and using the asymptotic formula (\ref{Ktendinfinity}) for $K_{\nu}(x)$ we can see that $p'(x)$ is also of order  $x^{\nu-1/2}\mathrm{e}^{-(1-\beta)x}$ as $x\rightarrow\infty$.  Hence, $xp(x)f'(x)$, $p(x)f(x)$, $xp(x)f(x)$ and $xp'(x)f(x)$ are equal to $0$ in the limit $x\rightarrow\infty$.  The terms $xp(x)f'(x)$ and $xp(x)f(x)$  are also equal to $0$ at the origin, because $f$ and $f'$ are continuous and thus bounded at the origin.  Hence, $I_1$ simplifies to
\begin{equation}\label{i11}I_1= -\lim_{x\downarrow 0}\{(2\nu +2\beta x)p(x)-xp'(x)\}f(x).
\end{equation}
Using formula (\ref{cat}) to differentiate $K_{\nu}(x)$ gives
\begin{align*}I_1&= -\kappa_{\nu,\beta}\lim_{x\downarrow 0}f(x)\{(2\nu+2\beta x)x^{\nu}\mathrm{e}^{\beta x}K_{\nu}(x)-x^{\nu}\mathrm{e}^{\beta x}(xK_{\nu}'(x)+\nu K_{\nu}(x)+\beta xK_{\nu}(x))\} \\
&=-\kappa_{\nu,\beta}\lim_{x\downarrow 0}x^{\nu}f(x)\{-xK_\nu'(x)+(\nu+\beta x)K_{\nu}(x)\}\\
&=-\kappa_{\nu,\beta}\lim_{x\downarrow 0}x^{\nu}f(x)\{\tfrac{1}{2}x(K_{\nu+1}(x)-K_{\nu-1}(x))+(\nu+\beta x)K_{\nu}(x)\}.
\end{align*}

We now calculate the limit in the above expression.  We first consider the case $\nu > 0$.  Applying the asymptotic formula (\ref{Ktend0}) gives  
\[I_1=-\kappa_{\nu,\beta}\lim_{x\downarrow 0}\{2^{\nu-1}\Gamma(\nu+1)+2^{\nu-1} \nu \Gamma(\nu)\}f(x)=-\kappa_{\nu,\beta}\lim_{x\to 0^{+}}2^{\nu}\Gamma(\nu+1)f(x),\] 
since $\nu\Gamma(\nu)=\Gamma(\nu+1)$.  Now consider the case $\nu =0$.  We use the fact that $K_{1}(x)=K_{-1}(x)$ to obtain
\[I_1=-\kappa_{0,\beta}\lim_{x\downarrow 0}f(x)xK_1(x)=-\kappa_{0,\beta}\lim_{x\to 0^{+}}\Gamma(1)f(x)=-\kappa_{0,\beta}\lim_{x\to 0^{+}}2^{0}\Gamma(1+1)f(x),\] 
since $\Gamma(1)=\Gamma(2)$.  Therefore we have
\[I_1=-\kappa_{\nu,\beta}\lim_{x\downarrow 0}2^{\nu}\Gamma(\nu+1)f(x) \quad \mbox{for all } \nu \geq 0.\]
Finally, we consider the case $-1/2<\nu<0$.  We use the fact that $K_{-\lambda}(x)=K_{\lambda}(x)$ to obtain  
\begin{align*}I_1&=-\kappa_{\nu,\beta}\lim_{x\downarrow 0}\{\tfrac{1}{2}x^{\nu+1}(K_{\nu+1}(x)+K_{1-\nu}(x))+\nu x^{\nu}K_{-\nu}(x)\}f(x) \\
&=-\kappa_{\nu,\beta}\lim_{x\downarrow 0}\{2^{\nu-1}\Gamma(\nu+1)+2^{\nu-1}(\Gamma(1-\nu)-(-\nu)\Gamma(-\nu))x^{2\nu}\}f(x) \\
&=-\kappa_{\nu,\beta}\lim_{x\downarrow 0}2^{\nu-1}\Gamma(\nu+1)f(x) \quad \mbox{for }-1/2<\nu<0.
\end{align*}

A similar argument (with the difference being that here $p(x)=\kappa_{\nu,\beta}(-x)^{\nu}\mathrm{e}^{\beta x}K_{\nu}(-x)$) shows that
\begin{equation*}I_2=\begin{cases} \kappa_{\nu,\beta}\lim_{x\uparrow 0}2^{\nu}\Gamma(\nu+1)f(x), & \quad \nu\geq0, \\
\kappa_{\nu,\beta}\lim_{x\uparrow 0}2^{\nu-1}\Gamma(\nu+1)f(x), & \quad -1/2<\nu<0. \end{cases} \\
\end{equation*}
As $f$ is continuous, it follows that $I_1=-I_2$ (and so $I_1+I_2=0$), which completes the proof of necessity.

\emph{Sufficiency}.  For fixed $z \in \mathbb{R}$, let $f(x) := f_z(x)$ be a bounded solution to the differential equation
\begin{equation} \label{king} xf''(x)+(2\nu +1+2\beta x)f'(x)+((2\nu +1)\beta -(1-\beta^2)x)f(x)=\chi_{(-\infty,z]}(x)-K_{\nu,\beta}(z),
\end{equation}
where $K_{\nu,\beta}(z)$ is the cumulative distribution function of the $\mathrm{VG}_2(\nu,1,\beta,0)$ distribution.  Using Lemma \ref{forty} (below) with $h(x)=\chi_{(-\infty,z]}(x)$ we see that a solution to (\ref{king}) is given by
\begin{align*}f_z(x)&=-\frac{e^{-\beta x}K_{\nu}(|x|)}{|x|^{\nu}}\int_0^xe^{\beta y} |y|^{\nu}I_{\nu}(|y|)[\chi_{(-\infty,z]}(x) -K_{\nu,\beta}(z)]\,\mathrm{d}y \\
&\quad -\frac{e^{-\beta x}I_{\nu}(|x|)}{|x|^{\nu}}\int_x^{\infty}e^{\beta y} |y|^{\nu}K_{\nu}(|y|)[\chi_{(-\infty,z]}(x) -K_{\nu,\beta}(z)]\,\mathrm{d}y.
\end{align*}
This  solution and its first derivative are bounded (see Lemma \ref{forty}) and is piecewise twice differentiable.  As $f_z$ and $f_z'$ are bounded, they satisfy the condition (\ref{limf1}) (with $\alpha=1$) and $f_z''$ must also satisfy the condition because, from (\ref{king}),
\begin{align*}|xf_z''(x)|\leq |(2\nu +1+2\beta x)f_z'(x)|+|((2\nu +1)\beta -(1-\beta^2)x)f_z(x)|+2 \leq A+B|x|
\end{align*}
for some constants $A$ and $B$.  Hence, if (\ref{lemmaone}) holds for all piecewise twice continuously differentiable functions satisfying (\ref{limf1}) (with $\alpha=1$), then by (\ref{king}),
\begin{align*}0&=\mathbb{E}[Wf_z''(W)+(2\nu +1+2\beta W)f_z'(W)+((2\nu +1)\beta -(1-\beta^2)W)f_z(W)] \\
&=\mathbb{E}[\chi_{(-\infty,z]}(W) -K_{\nu,\beta}(z)]\\
&=\mathbb{P}(W\leq z)-K_{\nu,\beta}(z).
\end{align*}
Therefore $W$ has the $\mathrm{VG}_2(\nu,1,\beta,0)$ distribution.
\end{proof}

Lemma \ref{ghds} suggests the following Stein equation for the $\mathrm{VG}_2(\nu,\alpha,\beta,\mu)$ distribution:
\begin{equation} \label{muggers} (x-\mu)f''(x) + (2\nu + 1 +2\beta (x-\mu))f'(x) + ((2\nu+1)\beta - (\alpha^2 -\beta^2)(x-\mu))f(x) = h(x)-\widetilde{\mathrm{VG}}^{\nu,\alpha}_{\beta,\mu}h, \end{equation}
where $\widetilde{\mathrm{VG}}^{\nu,\alpha}_{\beta,\mu}h$ denotes the quantity $\mathbb{E}(h(X))$ for $X \sim \mathrm{VG}_2(\nu,\alpha,\beta,\mu)$.

In order to simplify the calculations of Section 3.2, we will make use of the Stein equation for the $\mathrm{VG}_2(\nu,1,\beta,0)$ distribution, where $-1<\beta<1$.  Results for the full parametrisation can then be recovered by making a simple linear transformation.  For the $\mathrm{VG}_2(\nu,1,\beta,0)$ distribution, the Stein equation (\ref{muggers}) reduces to
\begin{equation} \label{eighty} xf''(x) + (2\nu+1 +2\beta x)f'(x) + ((2\nu+1)\beta - (1 -\beta^2)x)f(x) = h(x) -\widetilde{\mathrm{VG}}^{\nu,1}_{\beta,0}h.
\end{equation}
Changing parametrisation in (\ref{muggers}) via (\ref{parameter}) and multiplying through by $\sigma^2$ gives the $\mathrm{VG}_1(r,\theta,\sigma,\mu)$ Stein equation (\ref{nice}), which we presented in the introduction.  

\begin{remark}\label{nice eqn777}\emph{The $\mathrm{VG}_1(r,\theta,\sigma,\mu)$ Stein equation has the interesting property of being a (true) second order linear differential equation.  Such Stein equations are uncommon in the literature, although Pek\"oz et al$.$ \cite{pekoz}, Gaunt \cite{gaunt pn} and Pike and Ren \cite{pike} have obtained similar operators for the Kummer densities, the product of two mean zero normals, and the Laplace distribution, respectively.  Gaunt \cite{gaunt pn} and Pike and Ren \cite{pike} used the method of variation of parameters (see Collins \cite{collins} for an account of the method) to solve their equations, whereas Pek\"oz et al$.$ used a substitution to turn their second order operator into a first order operator, which leads to a double integral solution.  We attempted to follow this approach but the double integral solution we obtained was rather complicated.  However, solving using variation of parameters lead to a representation of the solution (see Lemma \ref{forty}) that enabled us to obtain uniform bounds for the solution and its first four derivatives (see Lemma \ref{vgboundz} and Theorem \ref{vgderbound}).}

\emph{We could have obtained a first order Stein operator for the $\mathrm{VG}_1(r,\theta,\sigma,\mu)$ distributions using the density approach of Stein et al$.$ \cite{stein3}.  However, this approach would lead to an operator involving the modified Bessel function $K_{\nu}(x)$.  Using such a Stein equation to prove approximation results with standard coupling techniques would be difficult.  In contrast, our $\mathrm{VG}_1(r,\theta,\sigma,\mu)$ Stein equation  is much more amenable to the use of couplings, as we shall see in Section 4.  Pek\"oz et al$.$ \cite{pekoz} encountered a similar situation (the density approach would lead to an operator involving the Kummer function) and proceeded as we did by instead considering a second order operator with simple coefficients.}
\end{remark}

\subsubsection{Special cases of the Variance-Gamma Stein equation}

Here we note a number of interesting special cases of the $\mathrm{VG}_1(r,\theta,\sigma,\mu)$ Stein equation.  Whilst the Gamma distribution is not covered by Lemma \ref{ghds}, we note that letting $r=2s$, $\theta=(2\lambda)^{-1}$, $\mu=0$ and taking the limit $\sigma\rightarrow 0$ in (\ref{nice}) gives the Stein equation
\[\lambda^{-1}(xf'(x) +(s - \lambda x)f(x)) = h(x) -\mathrm{VG}^{2s,(2\lambda)^{-1}}_{0,0}h,\]
which, recalling (\ref{delight}), we recognise as the $\Gamma(s,\lambda)$ Stein equation (\ref{delight}) of Luk \cite{luk} (up to a multiplicative factor).  

We also note that a Stein equation for the $\mathrm{VG}_1(r,0,\sigma/\sqrt{r},\mu)$ distribution is
\[\frac{\sigma^2}{r}(x-\mu)f''(x)+\sigma^2f'(x)-(x-\mu)f(x)=h(x)-\mathrm{VG}_{\sigma/\sqrt{r},\mu}^{r,0}h,\]
which in the limit $r\rightarrow\infty$ is the classical $N(\mu,\sigma^2)$ Stein equation. 

Taking $r=1$, $\sigma=\sigma_X\sigma_Y$ and $\mu=0$ in (\ref{nice}) gives the following Stein equation for distribution of the product of independent $N(0,\sigma_X^2)$ and $N(0,\sigma_Y^2)$ random variables (see part (iii) of Proposition \ref{norgamlap}):
\begin{equation*}\sigma_X^2\sigma_Y^2xf''(x) + \sigma_X^2\sigma_Y^2f'(x) -xf(x) = h(x) -\mathrm{VG}^{1,0}_{\sigma_X\sigma_Y,0}h.
\end{equation*} 
This Stein equation is in agreement with the Stein equation for the product of two independent, zero mean normal random variables that was obtained by Gaunt \cite{gaunt pn}.

Finally, we deduce a Stein equation for the Laplace distribution.  Recalling part (ii) of Proposition \ref{norgamlap}, we have that $\mathrm{Laplace}(0,\sigma)=\mathrm{VG}_1(2,0,\sigma,0)$.  Thus, we deduce the following Stein equation for the Laplace distribution:
\begin{equation}\label{roblap}\sigma^2xf''(x)+2\sigma^2f'(x)-xf(x)=h(x)-\mathbb{E}h(X),
\end{equation}
where $X\sim\mathrm{Laplace}(0,\sigma)$.  Pike and Ren \cite{pike} have obtained an alternative Stein characterisation of the Laplace distribution, which leads to the  initial value problem
\begin{equation}\label{lappike} f(x)-\sigma^2f''(x)=h(x)-\mathbb{E}h(X), \qquad f(0)=0. 
\end{equation}
They have also solved (\ref{lappike}) and have obtained uniform bounds for the solution and its first three derivatives.  Their characterisation was obtained by a repeated application of the density method, and is similar to the characterisation for the Exponential distribution that results from the density method (see Stein et al$.$ \cite{stein3}, Example 1.6), which leads to the Stein equation  
\begin{equation}\label{delight1}f'(x)-\lambda f(x)+\lambda f(0+)=h(x)-\mathbb{E}h(Y),
\end{equation} 
where $Y\sim\mathrm{Exp}(\lambda)$.  Since $\mathrm{Exp}(\lambda)=\Gamma(1,\lambda)$, equation (\ref{delight1}) and the Gamma Stein equation (\ref{delight}) (with $r=1$) give a choice of Stein equations for applications involving the Exponential distribution.  Both equations have been shown to be effective in the study of Exponential approximation, but in certain situations one equation may prove to be more useful than the other; see, for example, Pickett \cite{pickett} for a utilisation of (\ref{delight}), and  Pek\"oz and R\"ollin \cite{pekoz1} for an application involving (\ref{delight1}).  We would expect a similar situation to occur with the Laplace Stein equations (\ref{roblap}) and (\ref{lappike}), although we do not further investigate the use of these Stein equations in Laplace approximation. 

\subsubsection{Applications of Lemma \ref{ghds}}
The main application of Lemma \ref{ghds} that is considered in this paper involves the use of the resulting Stein equation in the proofs of the limit theorems of Section 4.  There are, however, other interesting results that follow from Lemma \ref{ghds}.  We consider a couple here.  

Suppose $W\sim\mathrm{VG}_2(\nu,\alpha,\beta,0)$.  Then taking $f(x)=\mathrm{e}^{tx}$, where $|t+\beta|<\alpha$ (which ensures that condition (\ref{limf1}) is satisfied), in the charactering equation (\ref{lemmaone}) and setting $M(t)=\mathbb{E}[\mathrm{e}^{tW}]$, we deduce that $M(t)$ satisfies the differential equation 
\begin{equation*}(t^2+2\beta t-(\alpha^2-\beta^2))M'(t)+(2\nu+1)(t+\beta)M(t)=0.
\end{equation*}
Solving this equation subject to the condition $M(0)=1$ then gives that the moment generating function of the Variance-Gamma distribution with $\mu=0$ is
\begin{equation*}M(t)=\bigg(\frac{\alpha^2-\beta^2}{\alpha^2-(\beta+t)^2}\bigg)^{\nu+1/2}=(1-2\theta t+\sigma^2t^2)^{-r/2}.
\end{equation*}

Similarly, taking $f(x)=x^k$ and setting $M_{k}=\mathbb{E}W^k$ leads to the following recurrence equation for the moments of the Variance-Gamma distributions with $\mu=0$:
\begin{equation*}(\alpha^2-\beta^2)M_{k+1}-\beta(2k+2\nu+1)M_k-k(2\nu+k)M_{k-1}=0,
\end{equation*}
which in terms of the first parametrisation is 
\begin{equation*}M_{k+1}-\theta(2k+r)M_k-\sigma^2k(r+k-1)M_{k-1}=0.
\end{equation*}
We have that $M_0=1$ and $M_1=(2\nu+1)\beta/(\alpha^2-\beta^2)=r\theta$ (see (\ref{vgmoment})), and thus we can solve these recurrence equations by forward substitution to obtain the moments of the Variance-Gamma distributions.  As far as the author is aware, these recurrence equations are new, although Scott et al$.$ \cite{scott} have already established a formula for the moments of general order of the Variance-Gamma distributions.

\subsection{Smoothness estimates for the solution of the Stein equation}

We now turn our attention to solving the $\mathrm{VG}_2(\nu,1,\beta,0)$ Stein equation (\ref{eighty}).  Handling this particular set of restricted parameters simplifies the calculations and allows us to write down the solution of $\mathrm{VG}_1(r,\theta,\sigma,\mu)$ after a straightforward change of variables.  

Since the homogeneous version of the $\mathrm{VG}_2(\nu,1,\beta,0)$ Stein equation has a simple fundamental system of solutions (see the proof of Lemma \ref{forty} in Appendix A), we consider variation of parameters to be an appropriate method of solution.  We carry out these calculations in Appendix A and present the solution in Lemma \ref{forty}.  We could also have solved the Stein equation by using generator theory.  Multiplying both sides of (\ref{eighty}) by $\frac{1}{x}$, we recognise the left-hand side of the equation as the generator of a Bessel process with drift with killing (for an account of the Bessel process with drift see Linetsky \cite{linetsky}).  The Stein equation can then be solved using generator theory (see Durrett \cite{durrett}, pp$.$ 249).  For a more detailed account of the application of the generator approach to Stein's method for Variance-Gamma distributions see Gaunt \cite{gaunt thesis}.

In the following lemma we give the solution to the Stein equation.  The proof is given in Appendix A.

\begin{lemma} \label{forty}
Let $h:\mathbb{R} \rightarrow \mathbb{R}$ be a measurable function with $\mathbb{E}|h(X)|<\infty$, where $X \sim \mathrm{VG}_2(\nu,1,\beta,0)$, and $\nu>-1/2$ and $-1<\beta<1$.  Then a solution $f:\mathbb{R} \rightarrow \mathbb{R}$ to the Variance-Gamma Stein equation (\ref{eighty}) is given by
\begin{align} \label{ink} f(x) &=-\frac{e^{-\beta x} K_{\nu}(|x|)}{|x|^{\nu}} \int_0^x e^{\beta y} |y|^{\nu} I_{\nu}(|y|) [h(y)- \widetilde{\mathrm{VG}}_{\beta,0}^{\nu,1}h] \,\mathrm{d}y \nonumber \\
&\quad-\frac{e^{-\beta x} I_{\nu}(|x|)}{|x|^{\nu}} \int_x^{\infty} e^{\beta y} |y|^{\nu} K_{\nu}(|y|)[h(y)- \widetilde{\mathrm{VG}}_{\beta,0}^{\nu,1}h]\,\mathrm{d}y,
\end{align}
where the modified Bessel functions $I_{\nu}(x)$ and $K_{\nu}(x)$ are defined in Appendix B.  Suppose further that $h$ is bounded, then $f(x)$ and $f'(x)$ and are bounded for all $x\in\mathbb{R}$.  Moreover, this is the unique bounded solution when $\nu\geq 0$ and $-1<\beta<1$.
\end{lemma}

\begin{remark} \label{mark} \emph{The equality 
\begin{equation} \label{useful} \int_{-\infty}^{x}e^{\beta y}|y|^{\nu}K_{\nu}(|y|)[h(y)-\widetilde{\mathrm{VG}}^{\nu,1}_{\beta,0}h]\,\mathrm{d}y=-\int_{x}^{\infty}e^{\beta y}|y|^{\nu}K_{\nu}(|y|)[h(y)-\widetilde{\mathrm{VG}}^{\nu,1}_{\beta,0}h]\,\mathrm{d}y
\end{equation}
is very useful when it comes to obtaining smoothness estimates for the solution to the Stein equation.  The equality ensures that we can restrict our attention to bounding the derivatives in the region $x\geq 0$, provided we obtain these bounds for both positive and negative $\beta$. }
\end{remark}

By direct calculations it is possible to bound the derivatives of the solution of the $\mathrm{VG}_2(\nu,1,\beta,0)$ Variance-Gamma Stein equation (\ref{eighty}).  Gaunt \cite{gaunt thesis} carried out these (rather lengthy) calculations for case $\beta=0$, to obtain uniform bounds on the solution of the Stein equation (\ref{eighty}) and its first four derivatives.  By a change of variables it is then possible to establish smoothness estimates for the solution of the $\mathrm{VG}_1(r,0,\sigma,\mu)$ Stein equation (\ref{nice}).  

Bounds on the first four derivatives of the solution of the $\mathrm{VG}_1(r,0,\sigma,\mu)$ Stein equation are sufficient for the limit theorems of Section 4.  However, it would be desirable to extend these bounds to the general case of the $\mathrm{VG}_1(r,\theta,\sigma,\mu)$ Stein equation.  Another open problem is to obtain uniform bounds for the derivatives of all order for the solution of the $\mathrm{VG}_1(r,0,\sigma,\mu)$ Stein equation.  This has been achieved for the normal and Gamma Stein equations (see the bounds of Goldstein and Rinott \cite{goldstein1} and Luk \cite{luk})) using the generator approach.  Gaunt \cite{gaunt thesis} made some progress towards the goal of achieving such bounds via the generator approach, but the problem remains unsolved. 

The following smoothness estimates for the solution of the $\mathrm{VG}_2(\nu,1,0,0)$ Stein equation were established by Gaunt \cite{gaunt thesis}.

\begin{lemma}\label{vgboundz}Let $\nu>-1/2$ and suppose that $h\in C_b^3(\mathbb{R})$.  Then the solution $f$, as given by (\ref{ink}), to the $\mathrm{VG}_2(\nu,1,0,0)$ Stein equation, and its first four derivatives are bounded as follows:
\begin{align*}\|f\|&\leq \bigg(\frac{1}{2\nu+1}+\frac{\pi\Gamma(\nu+1/2)}{2\Gamma(\nu+1)}\bigg)\|h-\widetilde{\mathrm{VG}}_{0,0}^{\nu,1}h\|, \\
\|f'\| &\leq \frac{2}{2\nu+1}\|h-\widetilde{\mathrm{VG}}_{0,0}^{\nu,1}h\|, \\
\|f''\| &\leq\bigg(\frac{\sqrt{\pi}}{2\sqrt{\nu+1/2}}+\frac{1}{2\nu+1}\bigg)\Big[3\|h'\|+4\|h-\widetilde{\mathrm{VG}}_{0,0}^{\nu,1}h\|\Big], \\
\|f^{(3)}\| &\leq \bigg(\frac{\sqrt{\pi}}{2\sqrt{\nu+1/2}}+\frac{1}{2\nu+1}\bigg)\Big[5\|h''\|+18\|h'\|+18\|h-\widetilde{\mathrm{VG}}_{0,0}^{\nu,1}h\|\Big]\\
&\quad+\frac{1}{v(\nu)}\|h-\widetilde{\mathrm{VG}}_{0,0}^{\nu,1}h\|, \\
\|f^{(4)}\| &\leq \bigg(\frac{\sqrt{\pi}}{2\sqrt{\nu+1/2}}+\frac{1}{2\nu+1}\bigg)\Big[8\|h^{(3)}\|+52\|h''\|+123\|h'\|+123\|h-\widetilde{\mathrm{VG}}_{0,0}^{\nu,1}h\|\Big]\\
&\quad+\frac{1}{v(\nu)}\Big[\|h'\|+\|h-\widetilde{\mathrm{VG}}_{0,0}^{\nu,1}h\|\Big],
\end{align*}
where $v(\nu)$ is given by
\begin{equation*}v(\nu)= \begin{cases} 2^{2\nu+1}\nu!(\nu+2)!(2\nu+1), & \quad \nu\in \mathbb{N}, \\
|\sin(\pi\nu)|2^{2\nu}\Gamma(\nu+1)\Gamma(\nu+4)(2\nu+1), & \quad \nu>-1/2 \mbox{ and } \nu\notin\mathbb{N}.\end{cases}
\end{equation*}
\end{lemma}
The bounds given in Lemma \ref{vgboundz} are of order $\nu^{-1/2}$ as $\nu\rightarrow\infty$, except when $2\nu$ is not equal to an integer, but is sufficiently close to an integer that 
\[\sin(\pi\nu)2^{2\nu}\Gamma(\nu+1)\Gamma(\nu+4)\sqrt{\nu}=o(1).\]
Gaunt \cite{gaunt thesis} remarked that the rogue $1/\sin(\pi\nu)$ term appeared to be an artefact of the analysis that was used to obtain the bounds.  

It is also worth noting that the bounds of Lemma \ref{vgboundz} break down as $\nu\rightarrow-1/2$.  This is to be expected, because in this limit the $\mathrm{VG}_2(\nu,1,0,0)$ distribution approaches a point mass at the origin. 

The bounds simplify in the case that $\nu\in\{0,1/2,1,3/2,\ldots\}$, and from these bounds we use a simple change of variables to obtain uniform bounds on the first derivatives of the solution of the $\mathrm{VG}_1(r,0,\sigma,\mu)$ Stein equation (\ref{nice}) for the case that $r$ is a positive integer.  These bounds are of order $r^{-1/2}$ as $r\rightarrow\infty$, which is the same order as Pickett's \cite{pickett} bounds (\ref{pickbound1}) for the solution of the $\Gamma(r,\lambda)$ Stein equation (\ref{delight}).  

\begin{theorem}\label{vgderbound}Let $r$ be a positive integer and let $\sigma>0$.  Suppose that $h\in C_b^3(\mathbb{R})$, then the solution of the $\mathrm{VG}_1(r,0,\sigma,\mu)$ Stein equation (\ref{nice}) and its derivatives up to fourth order satisfy 
\begin{equation*}\label{cbbd}\|f^{(k)}\|\leq M_{r,\sigma}^k(h), \qquad k=0,1,2,3,4,
\end{equation*}
where
\begin{eqnarray*}M_{r,\sigma}^0(h)&\leq& \frac{1}{\sigma}\bigg(\frac{1}{r}+\frac{\pi\Gamma(r/2)}{2\Gamma(r/2+1/2)}\bigg)\|h-\mathrm{VG}_{\sigma,\mu}^{r,0}h\|, \\
M_{r,\sigma}^1(h) &\leq& \frac{2}{\sigma^2r}\|h-\mathrm{VG}_{\sigma,\mu}^{r,0}h\|, \\
M_{r,\sigma}^2(h) &\leq& \frac{1}{\sigma^2}\bigg(\sqrt{\frac{\pi}{2r}}+\frac{1}{r}\bigg)\bigg[3\|h'\|+\frac{4}{\sigma}\|h-\mathrm{VG}_{\sigma,\mu}^{r,0}h\|\bigg], \\
M_{r,\sigma}^3(h) &\leq& \frac{1}{\sigma^2}\bigg(\sqrt{\frac{\pi}{2r}}+\frac{1}{r}\bigg)\bigg[5\|h''\| +\frac{18}{\sigma}\|h'\| +\frac{19}{\sigma^2}\|h-\mathrm{VG}_{\sigma,\mu}^{r,0}h\|\bigg],\\
M_{r,\sigma}^4(h)&\leq& \frac{1}{\sigma^2}\bigg(\sqrt{\frac{\pi}{2r}}+\frac{1}{r}\bigg)\bigg[8\|h^{(3)}\|+\frac{52}{\sigma}\|h''\| +\frac{124}{\sigma^2}\|h'\|+\frac{124}{\sigma^3}\|h-\mathrm{VG}_{\sigma,\mu}^{r,0}h\|\bigg],
\end{eqnarray*}  
and $f^{(0)}\equiv f$. 
\end{theorem}

\begin{proof}Let $g_{\tilde{h}}(x)$ denote the solution (\ref{ink}) to the $\mathrm{VG}_2(\nu,1,0,0)$ Stein equation (\ref{eighty})
\[xg''(x)+(2\nu+1)g'(x)-xg(x)=\tilde{h}(x)-\widetilde{\mathrm{VG}}_{0,0}^{\nu,1}\tilde{h}.\]
Then $f_h(x)=\frac{1}{\sigma}g_{\tilde{h}}(\frac{x-\mu}{\sigma})$ solves the $\mathrm{VG}_1(r,0,\sigma,\mu)$ Stein equation (\ref{nice})
\[\sigma^2(x-\mu)f''(x)+\sigma^2rf'(x)-(x-\mu)f(x)=h(x)-\mathrm{VG}_{\sigma,\mu}^{r,0} h,\]
where $r=2\nu+1$ and $h(x)=\tilde{h}(\frac{x-\mu}{\sigma})$, since $\mathrm{VG}_{\sigma,\mu}^{r,0} h = \widetilde{\mathrm{VG}}_{0,0}^{\nu,1} \tilde{h}$.  That $\mathrm{VG}_{\sigma,\mu}^{r,0} h = \widetilde{\mathrm{VG}}_{0,0}^{\nu,1} \tilde{h}$ is verified by the following calculation:
\begin{align*}\mathrm{VG}_{\sigma,\mu}^{r,0} h &=\int_{-\infty}^{\infty}\frac{1}{\sigma\sqrt{\pi}\Gamma(\frac{r}{2})}\bigg(\frac{|x-\mu|}{2\sigma}\bigg)^{\frac{r-1}{2}}K_{\frac{r-1}{2}}\bigg(\frac{|x-\mu|}{\sigma}\bigg)h(x)\,\mathrm{d}x \\
&=\int_{-\infty}^{\infty}\frac{1}{\sqrt{\pi}\Gamma(\nu+\frac{1}{2})}\bigg(\frac{|u|}{2}\bigg)^{\nu}K_{\nu}(|u|)\tilde{h}(u)\,\mathrm{d}u \\
&= \widetilde{\mathrm{VG}}_{0,0}^{\nu,1} \tilde{h},
\end{align*}
where we made the change of variables $u=\frac{x-\mu}{\sigma}$.  We have that $\|f_h^{(k)}\|=\sigma^{-k-1}\|g_{\tilde{h}}^{(k)}\|$ for $k\in\mathbb{N}$, and $\|\tilde{h}-\widetilde{\mathrm{VG}}_{0,0}^{\nu,1}\tilde{h}\|=\|h-\mathrm{VG}_{\sigma,\mu}^{r,0}h\|$ and $\|\tilde{h}^{(k)}\|=\sigma^k\|h^{(k)}\|$ for $k\geq 1$, and the result now follows from the bounds of Lemma \ref{vgboundz}.
\end{proof}

\section{Limit theorems for Symmetric-Variance Gamma distributions}

We now consider the Symmetric Variance-Gamma ($\theta =0$) limit theorem that we discussed in the introduction.  Let $\mathbf{X}$ be a $m \times r$ matrix of independent and identically random variables $X_{ik}$ with zero mean and unit variance.  Similarly, we let $\mathbf{Y}$ be a $n \times r$ matrix of independent and identically random variables $Y_{jk}$ with zero mean and unit variance, where the $Y_{jk}$ are independent of the $X_{ik}$.  Then the statistic
\begin{equation*}  W_r=\frac{1}{\sqrt{mn}}\sum_{i,j,k=1}^{m,n,r}X_{ik}Y_{jk}
\end{equation*}
is asymptotically $\mathrm{VG}_1(r,0,1,0)$ distributed, which can be seen by applying the central limit theorem and part (iv) of Proposition \ref{norgamlap}.  Pickett \cite{pickett} showed that the statistic $\frac{1}{m}\sum_{k=1}^d(\sum_{i=1}^mX_{ik})^2$, where the $X_{ik}$ are independent and identically random variables with zero mean, unit variance and bounded eighth moment, converges to a $\chi^2_{(d)}$ random variable at a rate of order $m^{-1}$ for smooth test functions.  We now exhibit a proof which gives a bound for the rate of convergence of the statistic $W_r$ to $\mathrm{VG}_1(r,0,1,0)$ random variables, under additional moment assumptions, which is shown to be of order $m^{-1}+n^{-1}$ for smooth test functions, using similar symmetry arguments to obtain this rate of convergence.    

\subsection{Local approach bounds for Symmetric Variance-Gamma distributions in the case $r=1$}

We first consider the case $r=1$; the general $r$ case follows easily as $W_r$ is a linear sum of independent $W_1$.  For ease of reading, in the statement of the following theorem and in its proof we shall set $X_i\equiv X_{i1}$, $Y_j \equiv Y_{j1}$ and $W \equiv W_1$.  Then we have the following:

\begin{theorem} \label{limit svg}
Suppose $X,X_1,\ldots,X_m$, $Y,Y_1,\ldots,Y_n$ are independent random variables with zero mean, unit variance and bounded sixth moment, with $X_i \stackrel{\mathcal{D}}{=}
X$ for all $i=1,\ldots, m$ and $Y_j \stackrel{\mathcal{D}}{=} Y$ for all $j=1,\ldots,n$.  Let $W=\frac{1}{\sqrt{mn}}\sum_{i,j=1}^{m,n}X_{i}Y_{j}$.  Then, for $h\in C_b^3(\mathbb{R})$, we have
\begin{equation} \label{svg bound} |\mathbb{E}h(W)-\mathrm{VG}^{1,0}_{1,0}h|\leq\gamma_1(X,Y)M_{1}^2(h)+\gamma_2(X,Y)M_{1}^3(h)+\gamma_3(X,Y)M_{1}^4(h),
\end{equation}
where the $M_{1}^k(h)$ are defined as in Theorem \ref{vgderbound}, $\mathrm{VG}^{1,0}_{1,0}h$ denotes the expectation of $h(Z)$ for $Z\sim \mathrm{VG}_1(1,0,1,0)$, and
\begin{align*}\gamma_{m,n}^1(X,Y) &= \frac{10}{n}|\mathbb{E}Y^3|\mathbb{E}|Y^3|+\frac{11}{\sqrt{mn}}|\mathbb{E}X^3|\mathbb{E}Y^4, 
\end{align*}
\begin{align*}
\gamma_{m,n}^2(X,Y) &=\frac{9}{m}\mathbb{E}X^4\mathbb{E}Y^4+\frac{30}{n}|\mathbb{E}Y^3|\mathbb{E}Y^4+\frac{85}{\sqrt{mn}}|\mathbb{E}X^3|\mathbb{E}|Y^5|+\frac{46}{\sqrt{mn}}\mathbb{E}|X^3||\mathbb{E}Y^3|\mathbb{E}Y^4, \\
\gamma_{m,n}^3(X,Y) &= \frac{1}{n}\mathbb{E}X^4\mathbb{E}Y^4(1+15|\mathbb{E}Y^3|)+\frac{284}{m}|\mathbb{E}X^3|\mathbb{E}|X^3|\mathbb{E}Y^6+\frac{148}{n}\mathbb{E}X^4|\mathbb{E}Y^3|\mathbb{E}|Y^5|\\
&\quad+\frac{135}{\sqrt{mn}}|\mathbb{E}X^3|\mathbb{E}X^4\mathbb{E}|Y^3|+\frac{248}{\sqrt{mn}}\mathbb{E}X^4|\mathbb{E}Y^3|.
\end{align*}
\end{theorem}

\begin{remark}\emph{Notice that the statistic $W=\frac{1}{\sqrt{mn}}\sum_{i,j=1}^{m,n}X_{i}Y_{j}$ is symmetric in $m$ and $n$ and the random variables $X_i$ and $Y_j$, and yet the bound (\ref{svg bound}) of Theorem \ref{limit svg} is not symmetric in $m$ and $n$ and the moments of $X$ and $Y$.  This asymmetry is a consequence of the local couplings that we used to obtain the bound. }    

\emph{In practice, when applying Theorem \ref{limit svg}, we would compute $\gamma_{m,n}^k(X,Y)$ for $k=1,2,3$, and $\gamma_{n,m}^k(Y,X)$ for $k=1,2,3$, which would yields two bounds for the quantity $\mathbb{E}h(W)-VG^{1,0}_{1,0}h$.  We would then take the minimum of these two bounds.  We proceed in this manner when applying bound (\ref{svg bound}) to prove Theorem \ref{binaryd2}.}
\end{remark}

Before proving Theorem \ref{limit svg}, we introduce some notation and preliminary lemmas.  We define the standardised sum $S$ and $T$ by
\begin{equation*}S=\frac{1}{\sqrt{m}}\sum_{i=1}^mX_i \qquad \text{and} \qquad T=\frac{1}{\sqrt{n}}\sum_{j=1}^nY_j
\end{equation*}
and we have that $W=ST$.  In our proof we shall make use of the sums 
\begin{equation*}S_i=S-\frac{1}{\sqrt{m}}X_i \qquad \text{and} \qquad T_j=T-\frac{1}{\sqrt{n}}Y_j
\end{equation*}
which are  independent of $X_{i}$ and  $Y_j$, respectively.  We therefore have the following formulas
\begin{eqnarray} \label{aones}  W-S_iT&=&ST - S_iT =\frac{1}{\sqrt{m}}X_iT \\
\label{bones} W-ST_j&=&ST - ST_j =\frac{1}{\sqrt{n}}Y_jS. \nonumber
\end{eqnarray} 

In the proof of Theorem \ref{svg bound} we use the following lemma, which can be found in Pickett \cite{pickett}, Lemma 4.3.   
\begin{lemma} \label{motor}  Let $X,X_1,\ldots,X_m$ be a collection of independent and identically distributed random variables with mean zero and unit variance.  Then, $\mathbb{E}S^p=O(1)$ for all $p\geq 1$.  Specifically,
\begin{eqnarray*}\mathbb{E}S^2&=&1, \\ 
\mathbb{E}S^4&=&\frac{1}{m}[3(m-1)+\mathbb{E}X^4]<3+\frac{\mathbb{E}X^4}{m}, \\ 
\mathbb{E}S^6&=&\frac{1}{m^2}[15(m-1)(m-2)+10(m-1)(\mathbb{E}X^3)^2+15(m-1)\mathbb{E}X^4+\mathbb{E}X^6]  \\ 
&<&15+\frac{10(\mathbb{E}X^3)^2}{m}+\frac{15\mathbb{E}X^4}{m}+\frac{\mathbb{E}X^6}{m^2}.
\end{eqnarray*}
and $\mathbb{E}|S|\leq(\mathbb{E}S^2)^{1/2}$, $\mathbb{E}|S^3|\leq(\mathbb{E}S^4)^{3/4}$, $\mathbb{E}|S^5|\leq(\mathbb{E}S^6)^{5/6}$, by H\"{o}lder's inequality.  
\end{lemma}
We will also use the following lemma.
\begin{lemma}\label{crossbike}Suppose $p\geq 1$, then $\mathbb{E}|S_i|^p\leq\mathbb{E}|S|^p$.
\end{lemma} 
\begin{proof}Applying Jensen's inequality gives
\[\mathbb{E}|S|^p=\mathbb{E}(\mathbb{E}(|S_i+n^{-1/2}X_i|^p \: | \: S_i)) 
\geq\mathbb{E}|\mathbb{E}(S_i+n^{-1/2}X_i \: | \: S_i)|^p 
=\mathbb{E}|S_i|^p,\]
as required.
\end{proof}

Using the $\mathrm{VG}_1(1,0,1,0)$ Stein equation (\ref{nicerr}) (with $r=1$), we require a bound on the expression $\mathbb{E}[Wf''(W)+f'(W)-Wf(W)].$  We split the proof into two parts.  In the first part of the proof, we use use local couplings and Taylor expansions to bound $\mathbb{E}[Wf''(W)+f'(W)-Wf(W)]$ by the remainder terms that result from our Taylor expansions.  Most of these terms are shown to be of the desired order of $O(m^{-1}+n^{-1})$, but the bounding of some of the terms is more involved.  The second part of the proof is devoted to bounding these terms to the required order. 

\subsubsection{Proof Part I: Expansions and Bounding} 

Due to the independence of the $X_i$ and $Y_j$ variables, we are in the realms of the local approach coupling.  We Taylor expand $f(W)$ about $S_{i}T$ to obtain
\begin{align*} \label{kent} &\mathbb{E}[Wf''(W)+f'(W)-Wf(W)]   \\
&=   \mathbb{E}STf''(W) +\mathbb{E}f'(W)-\frac{1}{\sqrt{m}}\sum_{i=1}^m\mathbb{E}X_iT\bigg(f(S_iT)+(ST-S_iT)f'(S_iT) \\
&\quad+\frac{1}{2}(ST-S_iT)^2f''(S_iT)+\frac{1}{6}(ST-S_iT)^3f^{(3)}(S_i^{[1]}T)\bigg),
\end{align*}
where $S_{i}^{[1]}=S_{i}+\theta_1(S-S_{i})$ for some $\theta_1\in(0,1)$.  Later in the proof we shall write $T_{j}^{[q]}=T_{j}+\theta_q(T-T_{j})$, where $\theta_q\in(0,1)$.  Using independence and the fact that $\mathbb{E}X_i=0$, we have
\[\sum_{i=1}^{m}\mathbb{E}X_iTf(S_{i}T)=\sum_{i=1}^{m}\mathbb{E}X_i\mathbb{E}Tf(S_{i}T)=0.\]
As $ST-S_iT=\frac{1}{\sqrt{m}}X_iT$, we obtain
\begin{equation*}\mathbb{E}\{Wf''(W)+f'(W) -Wf(W)\}=N_1 +R_1+R_2,
\end{equation*}
where
\begin{eqnarray*}
N_1 &=& \mathbb{E}STf''(W)+\mathbb{E}f'(W) 
-\frac{1}{m}\sum_{i=1}^m\mathbb{E}X_i^2T^2f'(S_iT),\\
R_1 &=& -\frac{1}{2m^{3/2}}\sum_{i=1}^m\mathbb{E}X_i^3T^3f''(S_iT), \\
R_2 &=& -\frac{1}{6m^2}\sum_{i=1}^m\mathbb{E}X_i^4T^4f^{(3)}(S_i^{[1]}T), 
\end{eqnarray*}

We begin by bounding $R_1$ and $R_2$.  Taylor expanding $f''(S_iT)$ about $W$ and using (\ref{aones}) gives
\begin{align*}|R_1|&=\frac{|\mathbb{E}X^3|}{2m^{3/2}}\bigg|\sum_{i=1}^m\mathbb{E}T^3f''(S_iT)\bigg| \\
&=\frac{|\mathbb{E}X^3|}{2m^{3/2}}\bigg|\sum_{i=1}^m\mathbb{E}T^3f''(W)-\frac{1}{\sqrt{m}}\sum_{i=1}^m\mathbb{E}X_iT^4f^{(3)}(S_i^{[2]}T)\bigg| \\
&\leq\frac{|\mathbb{E}X^3|}{2\sqrt{m}}|\mathbb{E}T^3f''(W)|+\frac{\|f^{(3)}\||\mathbb{E}X^3|}{2m}\bigg(3+\frac{\mathbb{E}Y^4}{n}\bigg),
\end{align*}
where we used that the random variables $X,X_1,\ldots X_m$ are identically distributed.  In obtaining the last inequality we used that $\mathbb{E}T^4<3+\frac{\mathbb{E}Y^4}{n}$ and that $\mathbb{E}|X_i|\leq \sqrt{\mathbb{E}X_i^2}=1$.  Bounding the term $\frac{1}{\sqrt{m}}|\mathbb{E}T^3f''(W)|$ to the desired order of $O(m^{-1}+n^{-1})$ is somewhat involved and is deferred until the part II of the proof.

The bound for $R_2$ is immediate.  We have
\[|R_2|\leq\frac{\|f^{(3)}\|}{6m^2}\sum_{i=1}^m\mathbb{E}X^4\mathbb{E}T^4\leq\frac{\|f^{(3)}\|}{6m}\mathbb{E}X^4\bigg(3+\frac{\mathbb{E}Y^4}{n}\bigg).\]

We now consider $N_1$.  We use independence and that $\mathbb{E}X_i^2=1$ and then Taylor expand $f'(S_iT)$ about $W$ to obtain
\begin{align*}\frac{1}{m}\sum_{i=1}^m\mathbb{E}X_i^2T^2f'(S_iT) 
&=\mathbb{E}T^2f'(W)-\frac{1}{m^{3/2}}\sum_{i=1}^m\mathbb{E}X_iT^3f''(W) \\
&\quad-\frac{1}{2m^2}\sum_{i=1}^m\mathbb{E}X_i^2T^4f^{(3)}(S_i^{[3]}T).
\end{align*}
Taylor expanding $f''(W)$ about $S_iT$ gives
\begin{align*}\frac{1}{m^{3/2}}\sum_{i=1}^m\mathbb{E}X_iT^3f''(W)&=\frac{1}{m^{3/2}}\sum_{i=1}^m\mathbb{E}X_iT^3f''(S_iT)+\frac{1}{m^{2}}\sum_{i=1}^m\mathbb{E}X_i^2T^4f^{(3)}(S_i^{[4]}T) \\
&=\frac{1}{m^{2}}\sum_{i=1}^m\mathbb{E}X_i^2T^4f^{(3)}(S_i^{[4]}T),
\end{align*}
where we used independence and that the $X_i$ have zero mean to obtain the final inequality.  Putting this together we have that
\[N_1=\mathbb{E}STf''(W)+\mathbb{E}f'(W)-\mathbb{E}T^2f'(W)+R_3,
\]
where 
\[|R_3|\leq\frac{1}{2m^{2}}\bigg|\sum_{i=1}^m\mathbb{E}X_i^2T^4f^{(3)}(S_i^{[3]}T)\bigg|+\frac{1}{m^{2}}\bigg|\sum_{i=1}^m\mathbb{E}X_i^2T^4f^{(3)}(S_i^{[4]}T)\bigg|  
\leq\frac{3\|f^{(3)}\|}{2m}\bigg(3+\frac{\mathbb{E}Y^4}{n}\bigg).
\]
Noting that $T^2=\frac{1}{\sqrt{n}}\sum_{j=1}^nY_jT=\frac{1}{\sqrt{n}}\sum_{j=1}^nY_j(\frac{1}{\sqrt{n}}Y_j+T_j)$, we may write $N_1$ as 
\[N_1=N_2+R_3+R_4,\]
where
\begin{eqnarray*}
N_2&=&\mathbb{E}STf''(W)-\frac{1}{\sqrt{n}}\sum_{j=1}^n\mathbb{E}Y_jT_jf'(W),\\
R_4&=&\mathbb{E}f'(W)-\frac{1}{n}\sum_{j=1}^n\mathbb{E}Y_j^2f'(W).
\end{eqnarray*}

We first consider $R_4$.  Taylor expanding $f'(W)$ about $ST_j$ and using that $ST-ST_j=\frac{1}{\sqrt{n}}Y_jS$ gives
\begin{align*}R_4&=\frac{1}{n}\sum_{j=1}^n\mathbb{E}(1-Y_j^2)f'(W) \\
&= \frac{1}{n}\sum_{j=1}^n\mathbb{E}(1-Y_j^2)\bigg(f'(ST_j)+\frac{1}{\sqrt{n}}Y_jSf''(ST_j)+\frac{1}{2n}Y_j^2S^2f^{(3)}(ST_j^{[5]})\bigg) \\
&=-\frac{\mathbb{E}Y^3}{n^{3/2}}\sum_{j=1}^n\mathbb{E}Sf''(ST_j)+\frac{1}{2n^2}\sum_{j=1}^n\mathbb{E}(Y_j^2-Y_j^4)S^2f^{(3)}(ST_j^{[5]}),
\end{align*}
where we used independence and that $\mathbb{E}Y_j=0$ and $\mathbb{E}Y_j^2=1$ to obtain the final equality.  Taylor expanding $f''(ST_j)$ about $W$ gives
\[\frac{\mathbb{E}Y^3}{n^{3/2}}\sum_{j=1}^n\mathbb{E}Sf''(ST_j)=\frac{\mathbb{E}Y^3}{\sqrt{n}}\mathbb{E}Sf''(W)-\frac{\mathbb{E}Y^3}{n^2}\sum_{j=1}^n\mathbb{E}Y_jS^2f^{(3)}(ST_j^{[6]}).\]
Putting this together we have the following bound for $R_4$:
\[|R_4|\leq\frac{|\mathbb{E}Y^3|}{\sqrt{n}}|\mathbb{E}Sf''(W)|+\frac{\|f^{(3)}\|}{2n}(1+2|\mathbb{E}Y^3|+\mathbb{E}Y^4).\]
As was the case with the term $\frac{1}{\sqrt{m}}|\mathbb{E}T^3f''(W)|$, bounding the quantity $\frac{1}{\sqrt{n}}|\mathbb{E}Sf''(W)|$ to the desired order of $O(m^{-1}+n^{-1})$ is somewhat involved and is deferred until part II of the  proof.

We now consider $N_2$.  Taylor expanding $f'(W)$ about $ST_j$, then using independence and that $\mathbb{E}Y_j=0$ and $\mathbb{E}Y_j^2=1$ gives
\begin{align*}N_2&=\mathbb{E}STf''(W)-\frac{1}{\sqrt{n}}\sum_{j=1}^n\mathbb{E}Y_jT_j\bigg(f'(ST_j)+\frac{1}{\sqrt{n}}Y_jSf''(ST_j) \\
&\quad+\frac{1}{2n}Y_j^2S^2f^{(3)}(ST_j)+\frac{1}{6n^{3/2}}Y_j^3S^3f^{(4)}(ST_j^{[7]})\bigg) \\
&=R_5+R_6+R_7,
\end{align*}
where
\begin{eqnarray*}
R_5&=&\frac{\mathbb{E}Y^3}{2n^{3/2}}\sum_{j=1}^n\mathbb{E}S^2T_jf^{(3)}(ST_j),\\
|R_6|&=&\frac{1}{6n^{2}}\bigg|\sum_{j=1}^n\mathbb{E}Y_j^4S^3T_jf^{(4)}(ST_j^{[7]})\bigg|\leq\frac{\|f^{(4)}\|}{6n}\mathbb{E}Y^4\bigg(3+\frac{\mathbb{E}X^4}{m}\bigg)^{3/4},\\
R_7&=&\mathbb{E}STf''(W)-\frac{1}{n}\sum_{j=1}^n\mathbb{E}ST_jf''(ST_j).
\end{eqnarray*}
Using independence and that the $Y_j$ have zero mean and then Taylor expanding $f^{(3)}(ST_j)$ about $W$ gives
\begin{align*}|R_5|&=\frac{|\mathbb{E}Y^3|}{2n^{3/2}}\bigg|\sum_{j=1}^n\mathbb{E}S^2Tf^{(3)}(ST_j)\bigg| \\
&=\frac{|\mathbb{E}Y^3|}{2n^{3/2}}\bigg|\sum_{j=1}^n\mathbb{E}\bigg[S^2Tf^{(3)}(W)-\frac{1}{\sqrt{n}}Y_jS^3\bigg(T_j+\frac{1}{\sqrt{n}}Y_j\bigg)f^{(4)}(ST_j^{[8]})\bigg]\bigg| \\
&\leq\frac{|\mathbb{E}Y^3|}{2\sqrt{n}}|\mathbb{E}S^2Tf^{(3)}(W)|+\frac{\|f^{(4)}\||\mathbb{E}Y^3|}{2n}\bigg(1+\frac{1}{\sqrt{n}}\bigg)\bigg(3+\frac{\mathbb{E}X^4}{m}\bigg)^{3/4}.
\end{align*}
The term $\frac{1}{\sqrt{n}}|\mathbb{E}S^2Tf^{(3)}(W)|$ is bounded to the required order of $O(m^{-1}+n^{-1})$ in part II of the proof.

To bound $R_7$ we Taylor expand $f''(W)$ about $ST_j$ and use independence and that the $Y_j$ have zero mean to obtain
\begin{align*}|R_7|&=\frac{1}{n}\bigg|\sum_{j=1}^n\mathbb{E}ST[f''(ST)-f''(ST_j)]\bigg| \\
&=\frac{1}{n}\bigg|\sum_{j=1}^n\mathbb{E}S(T_j+Y_j)\bigg(\frac{1}{\sqrt{n}}Y_jSf^{(3)}(ST_j)+\frac{1}{2n}Y_j^2S^2f^{(4)}(ST_j^{[9]})\bigg)\bigg| \\
&=\frac{1}{n}\bigg|\sum_{j=1}^n\mathbb{E}\bigg[\frac{1}{\sqrt{n}}Y_j^2S^2f^{(3)}(ST_j)+\frac{1}{2n}Y_j^2\bigg(T_j+\frac{1}{\sqrt{n}}Y_j\bigg)S^3f^{(4)}(ST_j^{[9]})\bigg]\bigg| \\
&\leq\frac{\|f^{(3)}\|}{n}+\frac{\|f^{(4)}\|}{2n}\bigg(3+\frac{\mathbb{E}X^4}{m}\bigg)^{3/4}\bigg(1+\frac{\mathbb{E}|Y^3|}{\sqrt{n}}\bigg).
\end{align*}

To summarise, at this stage we have shown that $|\mathbb{E}h(W)-\mathrm{VG}_{1,0}^{1,0}h|\leq\sum_{k=1}^7|R_k|$.  We have also bounded all terms to order $m^{-1}+n^{-1}$, except for the terms $\frac{1}{\sqrt{m}}|\mathbb{E}T^3f''(W)|$, $\frac{1}{\sqrt{n}}|\mathbb{E}Sf''(W)|$ and $\frac{1}{\sqrt{n}}|\mathbb{E}S^2Tf^{(3)}(W)|$.  In part II of the proof we shall use symmetry arguments to bound these terms to the required order.  But before doing so, we obtain a useful bound for $\mathbb{E}S^2Tf^{(3)}(W)$ that will ensure that our bound for $|\mathbb{E}h(W)-\mathrm{VG}_{1,0}^{1,0}h|$ will only involve bounds of the first four derivatives of the $\mathrm{VG}_1(1,0,1,0)$ Stein equation (\ref{nice}), and hence will only involve the supremum norm of the first three derivatives of the test function $h$.  The bound is given in the following lemma, which is proved in Appendix A.

\begin{lemma}\label{lemsec4}Let $f:\mathbb{R}\rightarrow\mathbb{R}$ be four times differentiable, then
\begin{align*}
|\mathbb{E}S^2Tf^{(3)}(W)|&\leq|\mathbb{E}ST^2f''(W)|+|\mathbb{E}Sf''(W)|+\frac{\|f^{(3)}\|}{\sqrt{n}}(1+\mathbb{E}|Y^3|) \\
&\quad+\frac{\|f^{(4)}\|}{2\sqrt{n}}\bigg(2+\frac{2}{\sqrt{n}}+\mathbb{E}|Y^3|+\frac{\mathbb{E}Y^4}{\sqrt{n}}\bigg)\bigg(3+\frac{\mathbb{E}X^4}{m}\bigg)^{3/4}.
\end{align*}
\end{lemma}

\subsubsection{Proof Part II: Symmetry Argument for Optimal Rate}
We now obtain bounds for $\frac{1}{\sqrt{n}}\mathbb{E}Sf''(W)$, $\frac{1}{\sqrt{n}}\mathbb{E}ST^2f''(W)$ and $\frac{1}{\sqrt{m}}\mathbb{E}T^3f''(W)$.  To obtain the desired rate of convergence we shall use symmetry arguments, which are similar to those used in Section 4.1.2$.$ of Pickett \cite{pickett} to achieve the optimal rate of convergence for chi-square limit theorems.

We begin by considering the bivariate standard normal Stein equation (see, for example, Goldstein and Rinott \cite{goldstein1}) with test functions $g_1(s,t)=sf''(st)$, $g_2(s,t)=st^2f''(st)$ and $g_3(s,t)=t^3f''(st)$.  The bivariate standard normal Stein equation with test function $g_k(s,t)$, $k=1,2,3,$ and solution $\psi_k$ is given by
\begin{equation} \label{notts} \frac{\partial^2\psi_k}{\partial s^2}(s,t)+\frac{\partial^2 \psi_k}{\partial t^2}(s,t)-s\frac{\partial \psi_k}{\partial s}(s,t) -t\frac{\partial \psi_k}{\partial t}(s,t) = g_k(s,t) -\mathbb{E}g_k(Z_1,Z_2),
\end{equation} 
where $Z_1$ and $Z_2$ are independent standard normal random variables.  

For large $m$ and $n$ we have $S \approx N(0,1)$ and $T \approx N(0,1)$, so we can apply the $O(m^{-1/2}+n^{-1/2})$ bivariate central limit convergence rate (see, for example, Reinert and R\"ollin \cite{reinert 1}) to bound the quantities $|\mathbb{E}g_k(S,T)-\mathbb{E}g_k(Z_1,Z_2)|$, $k=1,2,3$.  However, as the test functions $g_k$ are odd functions ($g_k(s,t)=-g_k(-s,-t)$ for all $s,t\in\mathbb{R}$), the following lemma ensures that $\mathbb{E}g_k(Z_1,Z_2)=0$, meaning that it should be possible to bound the expectations $\mathbb{E}g_k(S,T)$ to order $m^{-1/2}+n^{-1/2}$, which would yield $O(m^{-1}+n^{-1})$ bounds for  $\frac{1}{\sqrt{n}}\mathbb{E}Sf''(W)$, $\frac{1}{\sqrt{n}}\mathbb{E}ST^2f''(W)$ and $\frac{1}{\sqrt{m}}\mathbb{E}T^3f''(W)$. 

\begin{lemma} \label{moss}
Suppose $g(x,y) = -g(-x,-y)$, then $\mathbb{E}g(Z_1,Z_2) =0$ where $Z_1$, $Z_2$ are independent standard normal random variables.  In particular, if $Z_1$ and $Z_2$ are independent standard normal random variables, then $\mathbb{E}g_k(Z_1Z_2)=0$, for $k=1,2,3$.
\end{lemma}

\begin{proof}
Let $Z'_1=-Z_1$ and $Z'_2=-Z_2$.  Then $Z'_1 \stackrel{\mathcal{D}}{=} Z_1$ and  $Z_2 \stackrel{\mathcal{D}}{=} Z'_2$, so $\mathbb{E}g(Z_1,Z_2) = -\mathbb{E}g(Z'_1,Z'_2) = -\mathbb{E}g(Z_1,Z_2)$, and therefore $\mathbb{E}g(Z_1,Z_2) =0$.
\end{proof}
We now apply Lemma \ref{moss} and then perform Taylor expansions to bound the expectations $\mathbb{E}g_k(S,T)$.  Providing that a solution $\psi_k$ exists for the test function $g_k$, we have
\begin{align*}\mathbb{E}g_k(S,T) &= \mathbb{E}\bigg\{\frac{\partial^2\psi_k}{\partial s^2}(S,T)+\frac{\partial^2 \psi_k}{\partial t^2}(S,T)-S\frac{\partial \psi_k}{\partial s}(S,T) -T\frac{\partial \psi_k}{\partial t}(S,T)\bigg\} \\
&= R_8^k+R_9^k+R_{10}^k+R_{11}^k,
\end{align*}
where
\begin{eqnarray*}R_8^k&=&\frac{1}{2m^{3/2}}\sum_{i=1}^m\mathbb{E}X_i^3\frac{\partial^3\psi_k}{\partial s^3}\bigg(S_i+\phi_1\frac{X_i}{\sqrt{m}},T\bigg), \\
R_9^k &=& \frac{1}{2n^{3/2}}\sum_{j=1}^n\mathbb{E}Y_j^3\frac{\partial^3\psi_k}{\partial t^3}\bigg(S,T_j+\phi_2\frac{Y_j}{\sqrt{n}}\bigg), \\ 
R_{10}^k&=&\frac{1}{m^{3/2}}\sum_{i=1}^m\mathbb{E}X_i\frac{\partial^3\psi_k}{\partial s^3}\bigg(S_i+\phi_3\frac{X_i}{\sqrt{m}},T\bigg), \\
R_{11}^k &=& \frac{1}{n^{3/2}}\sum_{j=1}^n\mathbb{E}Y_j\frac{\partial^3\psi_k}{\partial t^3}\bigg(S,T_j+\phi_4\frac{Y_j}{\sqrt{n}}\bigg),
\end{eqnarray*}
with $\phi_1$, $\phi_2$, $\phi_3$, $\phi_4$ $\in (0,1)$.

Before we bound the remainder terms, we need bounds for the third order partial derivatives of the solution $\psi_k$ in terms of the derivatives of $f$.  We achieve this task by using the following lemma, the proof of which is given in Appendix A.  Before stating the lemma, we define the double factorial function.  The double factorial of a positive integer $n$ is given by
\begin{equation}n!!=\begin{cases} 1\cdot 3\cdot 5\cdot \cdots (n-2)\cdot n, & \text{$n>0$ odd,} \\
2\cdot 4\cdot 6 \cdots (n-2)\cdot n, & \text{$n>0$ even,}
\end{cases}
\end{equation}
and we define $(-1)!!=0!!=1$ (Arfken \cite{arfken}, p.547).

\begin{lemma}\label{appe33}Suppose that $f:\mathbb{R}^2\rightarrow\mathbb{R}$ is four times differentiable and let $g(s,t)=s^at^bf''(st)$, where $a,b\in\mathbb{N}$.  Then, the third order partial derivatives of the solution $\psi$ to the standard bivariate normal Stein equation (\ref{notts}) with test function $g$ are bounded as follows
\begin{eqnarray} \label{2ones} \bigg|\frac{\partial^3\psi}{\partial s^3}\bigg|&\leq&\frac{\pi}{4}\{2^{a+b}\|f^{(4)}\|(|s|^a+a!!)(|t|^{b+2}+(b+1)!!) \nonumber \\
&&+a2^{a+b-1}\|f^{(3)}\|(|s|^{a-1}+(a-1)!!)(|t|^{b+1}+b!!) \nonumber \\
&&+a(a-1)2^{a+b-4}\|f''\|(|s|^{a-2}+(a-2)!!)(|t|^b+(b-1)!!)\}, \\
\bigg|\frac{\partial^3\psi}{\partial t^3}\bigg|&\leq&\frac{\pi}{4}\{2^{a+b}\|f^{(4)}\|(|s|^{a+2}+(a+1)!!)(|t|^{b}+b!!) \nonumber \\
&&+b2^{a+b-1}\|f^{(3)}\|(|s|^{a+1}+a!!)(|t|^{b-1}+(b-1)!!) \nonumber \\
\label{256ones}&&+b(b-1)2^{a+b-4}\|f''\|(|s|^{a-2}+(a-1)!!)(|t|^{b-2}+(b-2)!!)\}. 
\end{eqnarray}
\end{lemma}

With these bounds it is straightforward to bound the remainder terms.  The following lemma allows us to easily deduce bounds for the remainder terms $R_8^k$, $R_9^k$, $R_{10}^k$ and $R_{11}^k$, $k=1,2,3$.
\newpage

\begin{lemma} \label{4ones} Suppose that $f:\mathbb{R}^2\rightarrow\mathbb{R}$ is four times differentiable, and let $g(s,t)=s^at^bf''(st)$, where $a,b\in\mathbb{N}$, then
\begin{align*}
|R_8^k|&\leq\frac{\pi}{8\sqrt{m}}\bigg\{2^{a+b}\|f^{(4)}\|\bigg[2^{a-1}\bigg[\mathbb{E}|X|^3\mathbb{E}|S|^a+\frac{\mathbb{E}|X|^{a+3}}{m^{a/2}}\bigg]+a!!\mathbb{E}|X|^3\bigg] (\mathbb{E}|T|^{b+2}+(b+1)!!) \\
&\quad+a2^{a+b-1}\|f^{(3)}\|\bigg[2^{a-2}\bigg[\mathbb{E}|X|^3\mathbb{E}|S|^{a-1}+\frac{\mathbb{E}|X|^{a+2}}{m^{(a-1)/2}}\bigg]+(a-1)!!\mathbb{E}|X|^3\bigg] (\mathbb{E}|T|^{b+1}+b!!) \\
&\quad+a(a-1)2^{a+b-4}\|f''\|\bigg[2^{a-3}\bigg[\mathbb{E}|X|^3\mathbb{E}|S|^{a-2}+\frac{\mathbb{E}|X|^{a+1}}{m^{(a-2)/2}}\bigg]+(a-2)!!\mathbb{E}|X|^3\bigg] \\
&\quad\times(\mathbb{E}|T|^b+(b-1)!!)\bigg\}, \\
|R_9^k| &\leq \frac{\pi}{8\sqrt{n}}\bigg\{2^{a+b}\|f^{(4)}\|\bigg[2^{b-1}\bigg[\mathbb{E}|Y|^3\mathbb{E}|T|^b+\frac{\mathbb{E}|Y|^{b+3}}{n^{b/2}}\bigg]+b!!\mathbb{E}|Y|^3\bigg] (\mathbb{E}|S|^{a+2}+(a+1)!!) \\
&\quad+b2^{a+b-1}\|f^{(3)}\|\bigg[2^{b-2}\bigg[\mathbb{E}|Y|^3\mathbb{E}|T|^{b-1}+\frac{\mathbb{E}|Y|^{b+2}}{n^{(b-1)/2}}\bigg]+(b-1)!!\mathbb{E}|Y|^3\bigg] (\mathbb{E}|S|^{a+1}+a!!) \\
&\quad+b(b-1)2^{a+b-4}\|f''\|\bigg[2^{b-3}\bigg[\mathbb{E}|Y|^3\mathbb{E}|T|^{b-2}+\frac{\mathbb{E}|Y|^{b+1}}{n^{(b-2)/2}}\bigg]+(b-2)!!\mathbb{E}|Y|^3\bigg] \\
&\quad\times(\mathbb{E}|S|^a+(a-1)!!)\bigg\}.
\end{align*}
The bound for $R_{10}^k$ is similar to the bound for $R_8^k$ but with $\mathbb{E}X^p$ and $\mathbb{E}|X^p|$ replaced with $\mathbb{E}X^{p-2}$ and $\mathbb{E}|X^{p-2}|$ respectively.  The bound for $R_{11}^k$ is similar to the bound for $R_9^k$ but with $\mathbb{E}Y^p$ and $\mathbb{E}|Y^p|$ replaced with $\mathbb{E}Y^{p-2}$ and $\mathbb{E}|Y^{p-2}|$. respectively. 
\end{lemma}

\begin{proof}We prove that the bound for $R_8^k$ holds; the bound for $R_9^k$ then follows by symmetry.  We begin by defining $S_i^*=S_i+\frac{\phi_1}{\sqrt{m}}X_i$.  We note the following simple bound for $|S_i^*|^p$, for $p\geq 1$:
\begin{equation} \label{5ones} |S_i^*|^p=\bigg|S_i+\frac{\phi_1}{\sqrt{m}}X_i\bigg|^p\leq2^{p-1}\bigg(|S_i|^p+\frac{\phi_1^p}{m^{p/2}}|X_i|^p\bigg)\leq2^{p-1}\bigg(|S_i|^p+\frac{|X_i|^p}{m^{p/2}}\bigg).
\end{equation}
Using our bound (\ref{2ones}) for the third order partial derivative of $\psi$ with respect to $s$, we have
\begin{align*}|R_8^k|&=\frac{1}{2m^{3/2}}\bigg|\sum_{i=1}^m\mathbb{E}X_i^3\frac{\partial^3\psi}{\partial s^3}(S_i^*,T)\bigg| \\
&\leq\frac{\pi}{8m^{3/2}}\sum_{i=1}^m\mathbb{E}\bigg|X_i^3\bigg\{2^{a+b}\|f^{(4)}\|(|S_i^*|^a+a!!)(|T|^{b+2}+(b+1)!!) \\
&\quad+a2^{a+b-1}\|f^{(3)}\|(|S_i^*|^{a-1}+(a-1)!!)(|T|^{b+1}+b!!) \\
&\quad+a(a-1)2^{a+b-4}\|f''\|(|S_i^*|^{a-2}+(a-2)!!)(|T|^b+(b-1)!!)\}| 
\end{align*}
\begin{align*}
&\leq\frac{\pi}{8m^{3/2}}\sum_{i=1}^m\mathbb{E}\bigg|X_i^3\bigg\{2^{a+b}\|f^{(4)}\|\bigg[2^{a-1}\bigg[|S_i|^a+\frac{|X_i|^a}{m^{a/2}}\bigg]+a!!\bigg](|T|^{b+2}+(b+1)!!) \\
&\quad+a2^{a+b-1}\|f^{(3)}\|\bigg[2^{a-2}\bigg[|S_i|^{a-1}+\frac{|X_i|^{a-1}}{m^{(a-1)/2}}\bigg]+a!!\bigg](|T|^{b+1}+b!!) \\
&\quad+a(a-1)2^{a+b-4}\|f''\|\bigg[2^{a-3}\bigg[|S_i|^{a-2}+\frac{|X_i|^{a-2}}{m^{(a-2)/2}}\bigg]+a!!\bigg](|T|^b+(b-1)!!)\bigg\}\bigg|,
\end{align*}
where we used (\ref{5ones}) to obtain the final inequality.  Applying the triangle inequality, that $X_i$ and $S_i$ are independent and that, by Lemma \ref{crossbike}, $\mathbb{E}|S_i|^p\leq \mathbb{E}|S|^p$ gives the desired bound.  The final statement of the lemma is clear.
\end{proof}

We can bound $R_8^k$, $R_9^k$, $R_{10}^k$ and $R_{11}^k$ by using the bounds in Lemma \ref{4ones}.  We illustrate the argument by bounding $R_{11}^1$.  In this case we have $g_1(s,t)=sf''(st)$, that is $a=1$ and $b=0$.  We have
\begin{align*}|R_8^1|&\leq \frac{\pi}{8\sqrt{m}}\bigg\{2\|f^{(4)}\|\bigg(2|X^3|+\frac{\mathbb{E}X^4}{\sqrt{m}}\bigg)(\mathbb{E}T^2+1!!)+\|f^{(3)}\|(2\mathbb{E}|X^3|)(\mathbb{E}|T|+0!!)\bigg\} \\
&=\frac{\pi}{2\sqrt{m}}\bigg\{\|f^{(4)}\|\bigg(2\mathbb{E}|X^3|+\frac{\mathbb{E}X^4}{\sqrt{m}}\bigg)+\|f^{(3)}\|\mathbb{E}|X^3|\bigg\},
\end{align*}
where we used that $0!!=1!!=1$, and $\mathbb{E}|T|\leq\sqrt{\mathbb{E}T^2}=1$ to obtain the second equality.  Continuing in this manner gives the following bounds:
\begin{align*}|R_9^1|&\leq\frac{\pi\|f^{(4)}\|}{2\sqrt{n}}\mathbb{E}|Y^3|\bigg[2+\bigg(3+\frac{\mathbb{E}X^4}{m}\bigg)^{3/4}\bigg], \\
|R_8^2|&\leq \frac{\pi}{\sqrt{m}}\bigg\{\|f^{(4)}\|\bigg(2\mathbb{E}|X^3|+\frac{\mathbb{E}X^4}{\sqrt{m}}\bigg)\bigg(6+\frac{\mathbb{E}Y^4}{n}\bigg) +\|f^{(3)}\|\mathbb{E}|X^3|\bigg[2+\bigg(3+\frac{\mathbb{E}Y^4}{n}\bigg)^{3/4}\bigg]\bigg\}, \\
|R_9^2|&\leq \frac{2\pi}{\sqrt{n}}\bigg\{\|f^{(4)}\|\bigg(2\mathbb{E}|Y^3|+\frac{\mathbb{E}|Y^5|}{n}\bigg)\bigg(2+\bigg(3+\frac{\mathbb{E}X^4}{m}\bigg)^{3/4}\bigg) \\
&\quad+\|f^{(3)}\|\bigg[2\mathbb{E}|Y^3|+\frac{\mathbb{E}Y^4}{\sqrt{n}}\bigg]+\|f''\|\mathbb{E}|Y^3|\bigg\}, \\
|R_8^3|&\leq \frac{2\pi\|f^{(4)}\|}{\sqrt{m}}\mathbb{E}|X^3|\bigg[8+\bigg(15+\frac{10(\mathbb{E}Y^3)^2}{n}+\frac{15\mathbb{E}Y^4}{n^2}+\frac{\mathbb{E}Y^6}{n^3}\bigg)^{5/6}\bigg], \\
|R_9^3|&\leq \frac{\pi}{4\sqrt{n}}\bigg\{8\|f^{(4)}\|\bigg[\mathbb{E}|Y^3|\bigg[3+4\bigg(3+\frac{\mathbb{E}Y^4}{m}\bigg)^{3/4}\bigg]+\frac{\mathbb{E}Y^6}{m^{3/2}}\bigg] \\
&\quad+24\|f^{(3)}\|\bigg[2\mathbb{E}|Y^3|+\frac{\mathbb{E}|Y^5|}{n}\bigg]+3\|f''\|\bigg[2\mathbb{E}|Y^3|+\frac{\mathbb{E}Y^4}{\sqrt{n}}\bigg]\bigg\}.
\end{align*}
The bound for $R_{10}^k$ is similar to the bound for $R_8^k$ but with $\mathbb{E}X^p$ and $\mathbb{E}|X^p|$ replaced with $\mathbb{E}X^{p-2}$ and $\mathbb{E}|X^{p-2}|$ respectively.  The bound for $R_{11}^k$ is similar to the bound for $R_9^k$ but with $\mathbb{E}Y^p$ and $\mathbb{E}|Y^p|$ replaced with $\mathbb{E}Y^{p-2}$ and $\mathbb{E}|Y^{p-2}|$. respectively.

We have therefore been able to bound the terms $\frac{1}{\sqrt{n}}\mathbb{E}Sf''(W)$, $\frac{1}{\sqrt{n}}\mathbb{E}ST^2f''(W)$ and $\frac{1}{\sqrt{m}}\mathbb{E}T^3f''(W)$ to order $m^{-1/2}+n^{-1/2}$.  It therefore follows that the remainder terms $R_1,\ldots,R_7$ are of order $m^{-1}+n^{-1}$.  We showed in part I of the proof that $|\mathbb{E}h(W)-\mathrm{VG}_{1,0}^{1,0}h|\leq\sum_{k=1}^7|R_k|$, and so we have achieved the desired $O(m^{-1}+n^{-1})$ bound.  We can now sum up the remainder terms to obtain the following bound:
\[|\mathbb{E}h(W)-\mathrm{VG}^{1,0}_{1,0}h|\leq\tilde{\gamma}_{m,n}^1(X,Y)M_{1}^2(h)+\tilde{\gamma}_{m,n}^2(X,Y)M_{1}^3(h)+\tilde{\gamma}_{m,n}^3(X,Y)M_{1}^4(h),\]
where
\begin{align*}\tilde{\gamma}_{m,n}^1(X,Y) &= \frac{\pi|\mathbb{E}Y^3|}{n}(\mathbb{E}|Y^3|+2)+\frac{3\pi|\mathbb{E}X^3|}{8\sqrt{mn}}\bigg[4+2\mathbb{E}|Y^3|+\frac{2}{\sqrt{n}}+\frac{\mathbb{E}Y^4}{\sqrt{n}}\bigg], \\
\tilde{\gamma}_{m,n}^2(X,Y) &= \frac{1}{6m}(9+3|\mathbb{E}X^3|+\mathbb{E}X^4)\bigg(3+\frac{\mathbb{E}Y^4}{n}\bigg)+\frac{\pi|\mathbb{E}Y^3|}{n}\bigg[4+2\mathbb{E}|Y^3|+\frac{2}{\sqrt{n}}+\frac{\mathbb{E}Y^4}{\sqrt{n}}\bigg] \\
&\quad\frac{|\mathbb{E}Y^3|}{2n}(1+\mathbb{E}|Y^3|)+\frac{3\pi|\mathbb{E}X^3|}{\sqrt{mn}}\bigg[4+2\mathbb{E}|Y^3|+\frac{2|\mathbb{E}Y^3|}{n}+\frac{\mathbb{E}|Y^5|}{n}\bigg] \\
&\quad+\frac{\pi|\mathbb{E}Y^3|}{\sqrt{mn}}(2+\mathbb{E}|X^3|)\bigg(2+\bigg(3+\frac{\mathbb{E}Y^4}{n}\bigg)^{3/4}\bigg), \\
\tilde{\gamma}_{m,n}^3(X,Y) &= \frac{1}{24n}\bigg(3+\frac{\mathbb{E}X^4}{m}\bigg)^{3/4}\bigg(4\mathbb{E}Y^4+6|\mathbb{E}Y^3|\bigg[4+\mathbb{E}|Y^3|+\frac{4}{\sqrt{n}}+\frac{\mathbb{E}Y^4}{\sqrt{n}}\bigg]\bigg) \\
&\quad+\frac{\pi|\mathbb{E}X^3|}{m}(\mathbb{E}|X^3|+2)\bigg(8+\bigg(15+\frac{10(\mathbb{E}Y^3)^2}{n}+\frac{15\mathbb{E}Y^4}{n^2}+\frac{\mathbb{E}Y^6}{n^3}\bigg)^{5/6}\bigg) \\
&\quad+\frac{\pi|\mathbb{E}Y^3|}{4n}\bigg(2+\bigg(3+\frac{\mathbb{E}X^4}{m}\bigg)^{3/4}\bigg)\bigg(18+\bigg(9+\frac{8}{n}\bigg)\mathbb{E}Y^3|+\frac{4\mathbb{E}|Y^5|}{n}\bigg) \\
&\quad+\frac{\pi|\mathbb{E}X^3|}{\sqrt{mn}}(\mathbb{E}|Y^3|+2)\bigg(3+4\bigg(3+\frac{\mathbb{E}X^4}{m}\bigg)^{3/4}\bigg) \\
&\quad+\frac{5\pi|\mathbb{E}Y^3|}{4\sqrt{mn}}\bigg(4+2\mathbb{E}|X^3|+\frac{2}{\sqrt{m}}+\frac{\mathbb{E}X^4}{\sqrt{m}}\bigg)\bigg(6+\frac{\mathbb{E}Y^4}{n}\bigg).
\end{align*}
To complete the proof of Theorem \ref{limit svg} we simplify this bound by using that $m,n\geq 1$, as well as that for $a\geq b\geq 2$ we have $\mathbb{E}X^a\geq\mathbb{E}X^b\geq 1$, and then round all numbers up to the nearest integer.  Doing so leads to bound (\ref{svg bound}), as required. \hfill $\square$

\subsection{Extension to the case $r>1$}
For the case of $r>1$, we have the following generalisation of Theorem \ref{limit svg}:

\begin{theorem} \label{multi}
Suppose the $X_{ik}$ and $Y_{jk}$ are defined as before, each with bounded sixth moment.  Let $W_r=\frac{1}{\sqrt{mn}}\sum_{i,j,k=1}^{m,n,r}X_{ik}Y_{jk}$.  Then, for any positive integer $r$ and $h\in C_b^3(\mathbb{R})$, we have
\begin{equation}\label{thm_result4}|\mathbb{E}h(W_r)-\mathrm{VG}^{r,0}_{1,0}h|\leq r(\gamma_1(X,Y)M_{r,1}^2(h)+\gamma_2(X,Y)M_{r,1}^3(h)+\gamma_3(X,Y)M_{r,1}^4(h)),
\end{equation}
where the $M_{r,1}^i(h)$ are defined as in Theorem \ref{vgderbound}, $\mathrm{VG}^{r,0}_{1,0}h$ denotes the expectation of $h(Z)$ for $Z\sim \mathrm{VG}(r,0,1,0)$, and the $\gamma_i$ are as in Theorem \ref{limit svg}.
\end{theorem}

\begin{proof}Define $W_{(k)}=\frac{1}{\sqrt{mn}}\sum_{i,j=1}^{m,n}X_{ik}Y_{jk}$, so that $W_r=\sum_{k=1}^rW_{(k)}$.  Using the $\mathrm{VG}_1(r,0,1,0)$ Stein equation (\ref{nicerr}) we have
\begin{align*}\mathbb{E}h(W_r)-\mathrm{VG}^{r,0}_{1,0}h
&=\mathbb{E}\{W_rf''(W_r)+rf'(W_r)-W_rf(W_r)]  \\
&=\sum_{k=1}^r\mathbb{E}[W_{(k)}f''(W_r)+f'(W_r)-W_{(k)}f(W_r)]  
\end{align*}
\begin{align*}
&=\sum_{k=1}^r\mathbb{E}\Big[\mathbb{E}[W_{(k)}f''(W_r)+f'(W_r)-W_{(k)}f(W_r) \: | \: W_{(1)},\ldots,W_{(k-1)},W_{(k+1)},\ldots,W_{(r)}]\Big]. 
\end{align*}
Since $\|g^{(n)}(x+c)\|=\|g^{(n)}(x)\|$ for any constant $c$, we may use bound (\ref{svg bound}) from Theorem \ref{limit svg} and the bounds of Theorem \ref{vgderbound} for the derivatives of the solution of the $VG_1(r,0,1,0)$ Stein equation to bound the above expression, which yields (\ref{thm_result4}).
\end{proof}

\begin{remark}\emph{The terms $M_{r,1}^k(h)$, for $k=2,3,4$, are of order $r^{-1/2}$ as $r\rightarrow\infty$ (recall Theorem \ref{vgderbound}), and therefore the bound of Theorem \ref{multi} is of order $r^{1/2}(m^{-1}+n^{-1})$.  This in agreement with bound of Theorem 4.7 of Pickett \cite{pickett} for chi-square approximation, which is of order $d^{1/2}m^{-1}$.}
\end{remark}

\begin{remark} \label{undone} \emph{The premise that the test function must be smooth is vital, as a non smooth test function will enforce a square-root convergence rate (cf. Berry-Ess\'{e}en theorem).  Consider the following example in the case of a $\mathrm{VG}_1(1,0,1,0)$ random variable with test function $h \equiv \chi_{\{0\}}$.  Let $X_i$, $i=1, \ldots ,m=2k$ and $Y_j$, $j=1, \ldots ,n=2l$, be random variables taking values in the set $\{-1,1\}$ with equal probability.  Then $\mathbb{E}X_i=\mathbb{E}Y_j=0$, $\mathrm{Var} X_i=\mathrm{Var} Y_j =1$ and
\begin{align*}\mathbb{E}h(W)&=\mathbb{P}\Big(\sum\nolimits_{i,j} X_i Y_j=0\Big) \\
&= \mathbb{P}\Big(\sum\nolimits_i X_i=0\Big)+\mathbb{P}\Big(\sum\nolimits_j Y_j=0\Big)- \mathbb{P}\Big(\sum\nolimits_i X_i=0\Big)\mathbb{P}\Big(\sum\nolimits_j Y_j=0\Big)\\ 
&=\binom{2k}{k}\left(\frac{1}{2}\right)^{2k} +\binom{2l}{l}\left(\frac{1}{2}\right)^{2l}-\binom{2k}{k}\left(\frac{1}{2}\right)^{2k}\binom{2l}{l}\left(\frac{1}{2}\right)^{2l}\\
&\approx \frac{1}{\sqrt{\pi k}}+\frac{1}{\sqrt{\pi l}}-\frac{1}{\pi kl} = \sqrt{\frac{2}{\pi m}}+\sqrt{\frac{2}{\pi n}}-\frac{4}{\pi mn},
\end{align*}
by Stirling's approximation.  Furthermore, $\mathrm{VG}_{1,0}^{1,0}h =\mathbb{P}(\mathrm{VG}(1,0,1,0)=0)=0$, and hence the univariate bound (\ref{svg bound}) fails.}
\end{remark}

\subsection{Application: Binary Sequence Comparison}

We now consider a straightforward application of Theorem \ref{limit svg} to binary sequence comparison.  This example is a simple special case of a more general problem of word sequence comparison, which is of particular importance to biological sequence comparisons.  One way of comparing the sequences uses $k$-tuples (a sequence of letters of length $k$).  If two sequences are closely related, we would expect the $k$-tuple content of both sequences to be very similar. A statistic for sequence comparison based on $k$-tuple content, known as the $D_2$ statistic was suggested by Blaisdell \cite{blaisdell} (for other statistics based on $k$-tuple content see Reinert et al$.$ \cite{waterman}).  Letting $\mathcal{A}$ denote an alphabet of size $d$, and $X_{\mathbf{w}}$ and $Y_{\mathbf{w}}$ the number of occurrences of the word $\mathbf{w}\in\mathcal{A}^k$ in the first and second sequences, respectively, then the $D_2$ statistic is defined by
\[D_2=\sum_{\mathbf{w}\in\mathcal{A}^k}X_{\mathbf{w}}Y_{\mathbf{w}}.\] 

Due to the complicated dependence structure at both the local and global level (for a detailed account of the dependence structure see Reinert et al$.$ \cite{lothaire}) approximating the asymptotic distribution of $D_2$ is a difficult problem.  However, for certain parameter regimes $D_2$ has been shown to be asymptotically normal and Poisson; see Lippert et al$.$ \cite{lippert} for a detailed account of the asymptotic distributions of $D_2$ for different parameter values.  

We now consider the case of an alphabet of size $2$ with comparison based on the content of $1$-tuples.  We suppose that the sequences are of length $m$ and $n$.  We assume that the alphabet is $\{0,1\}$, and $\mathbb{P}(0 \mbox{ appears})=\mathbb{P}(1 \mbox{ appears})=\frac{1}{2}$.  Denoting the number of occurrences of $0$ in the two sequences by $X$ and $Y$, respectively, then
\[D_2=XY+(m-X)(n-Y).\]
Clearly, $X$ and $Y$ are independent binomial variables with expectation $\frac{m}{2}$ and $\frac{n}{2}$ respectively.  Since $\mathbb{E}X^2=\frac{m(m+1)}{4}$, it is easy to compute the mean and variance of $D_2$, which are given by
\[\mathbb{E}D_2=\frac{mn}{2}  \qquad \mbox{and} \qquad \mathrm{Var}D_2=\frac{mn}{4}.\] 
We now consider the standardised $D_2$ statistic,
\begin{eqnarray} \label{cox apple} W&=&\frac{D_2-\mathbb{E}D_2}{\sqrt{\mathrm{Var}D_2}} \\
&=&\frac{2}{\sqrt{mn}}\bigg(XY+(m-X)(n-Y)-\frac{mn}{2}\bigg) \nonumber \\
&=&\frac{2}{\sqrt{mn}}\bigg(2XY-mX-nY+\frac{mn}{2}\bigg) \nonumber \\
&=&\bigg(\frac{X-\frac{m}{2}}{\sqrt{\frac{m}{4}}}\bigg)\bigg(\frac{Y-\frac{n}{2}}{\sqrt{\frac{n}{4}}}\bigg). \nonumber
\end{eqnarray}
By the central limit theorem, $(X-\frac{m}{2})/\sqrt{\frac{m}{4}}$ and $(Y-\frac{n}{2})/\sqrt{\frac{n}{4}}$ are approximately $N(0,1)$ distributed.  Therefore $W$ has an approximate $\mathrm{VG}_1(1,0,1,0)$ distribution.  We now apply Theorem \ref{limit svg} to obtain a bound on the error, in a weak convergence setting, in approximating the standardised $D_2$ statistic by its limiting $\mathrm{VG}_1(1,0,1,0)$ distribution.
\newpage

\begin{theorem}\label{binaryd2}For uniform i.i.d$.$ binary sequences of lengths $m$ and $n$, the standardised $D_2$ statistic $W$, defined as in equation (\ref{cox apple}), based on $1$-tuple content is approximately $\mathrm{VG}_1(1,0,1,0)$ distributed.  Moreover, for $h\in C_b^3(\mathbb{R})$ the following bound on the error in approximating $W$ by its asymptotic distribution holds,
\begin{equation} \label{uyt}|\mathbb{E}h(W)-\mathrm{VG}_{1,0}^{1,0}h|\leq \min\{A,B\},
\end{equation}
where
\[A=\frac{9}{m}M_{1,1}^3(h)+\frac{1}{n}M_{1,1}^4(h) \quad \mbox{and} \quad B=\frac{9}{n}M_{1,1}^3(h)+\frac{1}{m}M_{1,1}^4(h),\]
where the $M_{1,1}^k(h)$ is defined as in Theorem \ref{cbbd}, and $\mathrm{VG}^{1,0}_{1,0}h$ denotes the expectation of $h(Z)$, for $Z\sim \mathrm{VG}_1(1,0,1,0)$.
\end{theorem}

\begin{proof}We can write the number of occurrences that letter $0$ occurs in the first sequence as $X=\sum_{i=1}^m\mathbb{I}_i$, where $\mathbb{I}_i$ is the indicator random variable that letter $0$ appears at position $i$ in the first sequence.  Similarly, the number of occurrences of letter $0$ in the second sequence is given by $Y=\sum_{j=1}^n\mathbb{J}_j$, where $\mathbb{J}_j$ is the indicator random variable that letter $0$ appears at position $j$ in the second sequence.   The standardised $D_2$ statistic $W$ may therefore be write as
\[W=\frac{D_2-\mathbb{E}D_2}{\sqrt{\mathrm{Var}D_2}}=\bigg(\frac{X-\frac{m}{2}}{\sqrt{\frac{m}{4}}}\bigg)\bigg(\frac{Y-\frac{n}{2}}{\sqrt{\frac{n}{4}}}\bigg)=\frac{1}{\sqrt{mn}}\sum_{i,j=1}^{m,n}X_iY_j,\]
where $X_i=2(\mathbb{I}_i-\frac{1}{2})$ and $Y_j=2(\mathbb{J}_j-\frac{1}{2})$.  The $X_i$ and $Y_j$ are all independent and have zero mean and unit variance.  We may therefore apply Theorem \ref{limit svg} with $\mathbb{E}X_i^3=\mathbb{E}Y_j^3=0$ and $\mathbb{E}X_i^4=\mathbb{E}Y_j^4=1$ to obtain bound (\ref{uyt}).
\end{proof}

\appendix

\section{Proofs from the text}

Here we prove the lemmas that we stated in the main text without proof. 

\subsection{Proof of Proposition \ref{norgamlap}}

For clarity, we restate the proposition.

\begin{proposition}(i) Let $\sigma>0$ and $\mu\in\mathbb{R}$ and suppose that $Z_r$ has the $\mathrm{VG}_1(r,0,\sigma/\sqrt{r},\mu)$ distribution.  Then $Z_r$ converges in distribution to a $N(\mu,\sigma^2)$ random variable in the limit $r\rightarrow\infty$.

(ii) Let $\sigma>0$ and $\mu\in\mathbb{R}$, then a $\mathrm{VG}_1(2,0,\sigma,\mu)$ random variable has the $\mathrm{Laplace}(\mu,\sigma)$ distribution with probability density function 
\begin{equation}\label{laplacepdf}p_{\mathrm{VG_1}}(x;2,0,\sigma,\mu) =\frac{1}{2\sigma} \exp\bigg(-\frac{|x-\mu|}{\sigma}\bigg), \quad x \in \mathbb{R}.
\end{equation}

(iii) Suppose that $(X,Y)$ has the bivariate normal distribution with correlation $\rho$ and marginals $X\sim N(0,\sigma_X^2)$ and $Y\sim N(0,\sigma_Y^2)$. Then the product $XY$ follows the $\mathrm{VG}_1(1,\rho\sigma_X\sigma_Y, \sigma_X\sigma_Y\sqrt{1-\rho^2},0)$ distribution.

(iv) Let $X_1,\ldots,X_r$ and $Y_1,\ldots,Y_r$ be independent standard normal random variables.  Then $\mu+\sigma\sum_{k=1}^rX_kY_k$ has the $\mathrm{VG}_1(r,0,\sigma,\mu)$ distribution.  As a special case we have that a Laplace random variable with density (\ref{laplacepdf}) has the representation $\mu+\sigma(X_1Y_1+X_2Y_2)$.

(v) The Gamma distribution is a limiting case of the Variance-Gamma distributions: for $r>0$ and $\lambda>0$, the random variable $X_{\sigma}\sim \mathrm{VG}_1(2r,(2\lambda)^{-1},\sigma,0)$ convergences in distribution to a $\Gamma(r,\lambda)$ random variable in the limit $\sigma\downarrow 0$.

(vi) Suppose that $(X,Y)$ follows a bivariate gamma distribution with correlation $\rho$ and marginals $X\sim \Gamma(r,\lambda_1)$ and $Y\sim \Gamma(r,\lambda_2)$.  Then the random variable $X-Y$ has the $\mathrm{VG}_1(2r,(2\lambda_1)^{-1}-(2\lambda_2)^{-1},(\lambda_1\lambda_2)^{-1/2}(1-\rho)^{1/2},0)$ distribution.

\end{proposition}

\begin{proof}(i) Let $X_1,X_2,\ldots$ and $Y_1,Y_2\ldots$ be independent standard normal random variables, and define $Z_r=\mu+\frac{\sigma}{\sqrt{r}}\sum_{i=1}^rX_iY_i$.  Then by Corollary \ref{vuelta} it follows that $Z_r\sim \mathrm{VG}_1(r,0,\sigma/\sqrt{r},\mu)$.  Moreover, the products $X_iY_i$, $i=1,2,\ldots$, are independent and identically random variables with mean zero and unit variance, and by the central limit theorem $\frac{1}{\sqrt{r}}\sum_{i=1}^rX_iY_i\stackrel{\mathcal{D}}{\rightarrow} N(0,1)$.  Hence, $Z_r\stackrel{\mathcal{D}}{\rightarrow} N(\mu,\sigma^2)$ as $r\rightarrow \infty$.  

(ii) This follows by applying the formula $K_{\frac{1}{2}}(x)=\sqrt{\frac{\pi}{2x}}e^{-x}$ to the density (\ref{vgdef}).

(iii) Let $\tilde{X}=\frac{X}{\sigma_X}\sim N(0,1)$ and $\tilde{Y}=\frac{Y}{\sigma_Y}\sim N(0,1)$, and define the random  variable $W$ by
\[W=\frac{1}{\sqrt{1-\rho^2}}(\tilde{Y}-\rho \tilde{X}).\]
It is straightforward to show that $W\sim N(0,1)$, and that $W$ and $\tilde{X}$ are jointly normally distributed with correlation $0$.  We can therefore express the product $Z=XY$ in terms of independent standard normal random variables $\tilde{X}$ and $W$ as follows
\[Z=XY=\sigma_X\sigma_Y \tilde{X}\tilde{Y}=\sigma_X\sigma_Y\tilde{X}(\sqrt{1-\rho^2} W+\rho \tilde{X})=\sigma_X\sigma_Y\sqrt{1-\rho^2}\tilde{X}W+\rho\sigma_X\sigma_Y \tilde{X}^2.\]
Hence, by Corollary \ref{vuelta}, we have that $Z\sim \mathrm{VG}_1(1,\rho\sigma_X\sigma_Y, \sigma_X\sigma_Y\sqrt{1-\rho^2},0)$. 

(iv) Taking $\theta=0$ in Corollary (\ref{vuelta}) leads to the general representation.  The representation for the Laplace distribution now follows from part (ii).

(v) This follows on letting $\sigma\rightarrow0$ in Proposition \ref{vgcharvg} and then using the fact that if $Y\sim\Gamma(\alpha,\beta)$ then $kY\sim\Gamma(\alpha,\beta/k)$.

(vi) Theorem 6 of Holm and Alouini \cite{holm} gives the following formula for the probability density function of $Z=U-V$:
\begin{align*}p_Z(x)&=\frac{|x|^{r-1/2}}{\Gamma(r)\sqrt{\pi}\sqrt{\beta_1\beta_2(1-\rho)}}\bigg(\frac{1}{(\beta_1+\beta_2)^2-4\beta_1\beta_2\rho}\bigg)^{\frac{2r-1}{4}}\exp\bigg(\frac{x}{2(1-\rho)}\bigg(\frac{1}{\beta_1}-\frac{1}{\beta_2}\bigg)\bigg) \\
&\quad\times K_{r-\frac{1}{2}}\bigg(|x|\frac{\sqrt{(\beta_1+\beta_2)^2-4\beta_1\beta_2\rho}}{2\beta_1\beta_2(1-\rho)}\bigg), \qquad x\in\mathbb{R},
\end{align*} 
where 
\[\beta_1=\frac{1}{\lambda_1} \qquad \text{and} \qquad \beta_2=\frac{1}{\lambda_2}.\]
We can write the density  of $Z$ as follows
\begin{align}p_Z(x)&=\frac{|x|^{r-1/2}}{\Gamma(r)\sqrt{\pi}\sqrt{\beta_1\beta_2(1-\rho)}}\bigg(\frac{1}{(\beta_1+\beta_2)^2-4\beta_1\beta_2\rho}\bigg)^{\frac{2r-1}{4}}\exp\bigg(x\cdot\frac{\frac{1}{2}(\beta_1-\beta_2)}{\beta_1\beta_2(1-\rho)}\bigg) \nonumber \\
&\label{vgpdfholm}\quad\times K_{r-\frac{1}{2}}\bigg(|x|\frac{\sqrt{[(\beta_1-\beta_2)/2]^2-\beta_1\beta_2(1-\rho)}}{\beta_1\beta_2(1-\rho)}\bigg), \qquad x\in\mathbb{R}.
\end{align} 
Comparing (\ref{vgpdfholm}) with the Variance-Gamma density function (\ref{vgdef}), we see that $Z$ has a $\mathrm{VG}_1(2r,\theta,\sigma,0)$ distribution, where $\theta$ and $\sigma$ are given by
\begin{eqnarray*}\theta &=&\frac{\beta_1-\beta_2}{2}=\frac{1}{2\lambda_1}-\frac{1}{2\lambda_2}, \\
\sigma &=&\sqrt{\beta_1\beta_2(1-\rho)}=\sqrt{\frac{1-\rho}{\lambda_1\lambda_2}},
\end{eqnarray*}
as required.
\end{proof}  

\subsection{Proof of Lemma \ref{forty}}
We begin by proving that there is at most one bounded solution to the Variance-Gamma Stein equation (\ref{eighty}) when $\nu\geq 0$.  Suppose $u$ and $v$ are solutions to the Stein equation that satisfy $\|u^{(k)}\|, \textrm{ } \|v^{(k)}\| < \infty$.  Define $w=u-v$.  Then $w$ satisfies $\|w^{(k)}\| = \|u^{(k)} - v^{(k)} \| \leq  \|u^{(k)}\| + \|v^{(k)}\| < \infty$, and is a solution to the following differential equation
\[ xw'' (x) + (2\nu+1 +2\beta x)w' (x) +((2\nu+1)\beta -(1-\beta^2)x)w(x) = 0.\]
This homogeneous differential equation has general solution
\[ w(x) = A\mathrm{e}^{-\beta x}x^{-\nu} K_{\nu} (x) + B\mathrm{e}^{-\beta x}x^{-\nu} I_{\nu} (x).\]
From the asymptotic formula (\ref{roots}) for $I_{\nu}(x)$, it follows that to have a bounded solution we must take $B=0$.  From the asymptotic formula (\ref{Ktend0}) for $K_{\nu}(x)$, we see that $w(x)$ has a singularity at the origin if $\nu\geq 0$.  Therefore if $\nu\geq 0$, then for $w(x)$ to be bounded we must take $A=0$, and therefore $w=0$ and so $u=v$.  

We now use variation of parameters (see Collins \cite{collins}) to solve the Stein equation equation (\ref{eighty}).  The method allows us to solve differential equations of the form
\[v''(x)+p(x)v'(x)+q(x)v(x)=g(x).\]
Suppose $v_1(x)$ and $v_2(x)$ are linearly independent solutions of the homogeneous equation 
\[v''(x)+p(x)v'(x)+q(x)v(x)=0.\] 
Then the general solution to the inhomogeneous equation is given by
\[v(x)=-v_1(x)\int_a^x\frac{v_2(t)g(t)}{W(t)}\,\mathrm{d}t+v_2(x)\int_b^x\frac{v_1(t)g(t)}{W(t)}\,\mathrm{d}t,
\]
where $a$ and $b$ are arbitrary constants and $W(t)=W(v_1,v_2)=v_1v_2'-v_2v_1'$ is the Wronskian.

It is easy to verify that a pair of linearly independent solutions to the homogeneous equation
\[f''(x)+\bigg(\frac{(2\nu +1)}{x} +2\beta \bigg)f'(x)+\bigg(\frac{(2\nu +1)\beta}{x} -(1-\beta^2)\bigg)f(x)=0\]
are $\mathrm{e}^{-\beta x}x^{-\nu}K_{\nu}(x)$ and $\mathrm{e}^{-\beta x}x^{-\nu}I_{\nu}(x)$.  However, we take $f_1(x)=\mathrm{e}^{-\beta x}|x|^{-\nu}K_{\nu}(|x|)$ and $f_2(x)=\mathrm{e}^{-\beta x}|x|^{-\nu}I_{\nu}(|x|)$ as our linearly independent solutions to the homogeneous equation.  It will become clear later why this is a more suitable basis of solutions to the homogeneous equation.  We now show that $f_1$ and $f_2$ are indeed linearly independent solutions to the homogeneous equation.  From (\ref{defI}) we have
\[\frac{I_{\nu}(|x|)}{|x|^{\nu}}=\frac{1}{|x|^{\nu}}\sum_{k=0}^{\infty}\frac{1}{\Gamma(\nu+k+1)k!}\left(\frac{|x|}{2}\right)^{\nu+2k}=\sum_{k=0}^{\infty}\frac{1}{\Gamma(\nu+k+1)k!}\left(\frac{x}{2}\right)^{2k}=\frac{I_{\nu}(x)}{x^{\nu}}.\]
Formula (\ref{imre}) states that $K_{\nu}(-x)=(-1)^{\nu}K_{\nu}(x)-\pi i I_{\nu}(x)$ and therefore
\[\frac{K_{\nu}(-x)}{(-x)^{\nu}}=\frac{K_{\nu}(x)}{x^{\nu}}-\frac{\pi i}{(-1)^{\nu}}\frac{I_{\nu}(x)}{x^{\nu}},\]
and so
\[\frac{\mathrm{e}^{-\beta x}K_{\nu}(|x|)}{|x|^{\nu}}=\frac{\mathrm{e}^{-\beta x}K_{\nu}(x)}{x^{\nu}}-\frac{\pi i}{(-1)^{\nu}}\frac{\mathrm{e}^{-\beta x}I_{\nu}(x)}{x^{\nu}}\chi_{(-\infty,0]}(x).\]
Since $\mathrm{e}^{-\beta x}x^{\nu}I_{\nu}(x)$ is a solution to the homogeneous equation that is linearly independent of $\mathrm{e}^{-\beta x}x^{\nu}K_{\nu}(x)$, it follows that $\mathrm{e}^{-\beta x}|x|^{\nu}K_{\nu}(|x|)$ is a solution to the homogeneous equation.

From (\ref{diffKii}) and (\ref{diffIii}) we have
\[\frac{\mathrm{d}}{\mathrm{d}x}\left(\frac{K_{\nu}(|x|)}{|x|^{\nu}}\right)=-\frac{K_{\nu+1}(|x|)}{x|x|^{\nu -1}}, \quad \frac{\mathrm{d}}{\mathrm{d}x}\left(\frac{I_{\nu}(|x|)}{|x|^{\nu}}\right)=\frac{I_{\nu+1}(|x|)}{x|x|^{\nu -1}},\]
and therefore
\[W(x)=\frac{\mathrm{e}^{-2\beta x}(I_{\nu}(|x|)K_{\nu+1}(|x|)+K_{\nu}(|x|)I_{\nu+1}(|x|))}{x|x|^{2\nu-1}}=\frac{\mathrm{e}^{-2\beta x}}{x|x|^{2\nu}},\]
where we used (\ref{wront}) to obtain the equality in the above display.  Therefore the general solution to the inhomogeneous equation is given by
\begin{align*}f(x)&=-\frac{\mathrm{e}^{-\beta x}K_{\nu}(|x|)}{|x|^{\nu}}\int_a^x\mathrm{e}^{\beta y} |y|^{\nu}I_{\nu}(|y|)[h(y)- \widetilde{\mathrm{VG}}_{\beta,0}^{\nu,1}h]\,\mathrm{d}y \\
&\quad+\frac{\mathrm{e}^{-\beta x}I_{\nu}(|x|)}{|x|^{\nu}}\int_b^x\mathrm{e}^{\beta y} |y|^{\nu}K_{\nu}(|y|)[h(y)- \widetilde{\mathrm{VG}}_{\beta,0}^{\nu,1}h]\,\mathrm{d}y.
\end{align*}

This solution is clearly bounded everywhere except possibly for $x=0$ or in the limits $x \rightarrow \pm \infty$.  We therefore choose $a$ and $b$ so that our solution is bounded at these points and thus for all real $x$.  To ensure the solution is bounded at the origin we must take $a=0$.  We choose $b$ so that the solution is bounded in the limits $x \rightarrow \pm \infty$.  If we take $b =\infty$ then we obtain solution (\ref{ink}).  Taking $b=-\infty$ would lead to the same solution (see Remark \ref{mark}).  

Solution (\ref{ink}) is a candidate bounded solution, and we now verify that this solution and its first derivative are indeed bounded for all $x\in\mathbb{R}$.  Straightforward calculations show that, for $x\geq 0$,
\begin{eqnarray}\|f\| &\leq&\|\tilde{h}\|\bigg|\frac{\mathrm{e}^{-\beta x} K_{\nu}(x)}{x^{\nu}}\int_0^{x} \mathrm{e}^{\beta y}y^{\nu} I_{\nu}(y)\,\mathrm{d}y \bigg| +\|\tilde{h}\|\bigg|\frac{\mathrm{e}^{-\beta x}I_{\nu}(x)}{x^{\nu}} \int_x^{\infty} \mathrm{e}^{\beta y}y^{\nu} K_{\nu}(y)\,\mathrm{d}y\bigg|, \nonumber \\
\|f'\| &\leq&\|\tilde{h}\|\bigg|\frac{\mathrm{d}}{\mathrm{d}x}\bigg(\frac{\mathrm{e}^{-\beta x}K_{\nu}(x)}{x^{\nu}}\bigg)\int_0^x \mathrm{e}^{\beta y}y^{\nu} I_{\nu}(y)\,\mathrm{d}y \bigg| \nonumber\\
&& +\|\tilde{h}\|\bigg|\frac{\mathrm{d}}{\mathrm{d}x}\bigg(\frac{\mathrm{e}^{-\beta x}I_{\nu }(x)}{x^{\nu}}\bigg) \int_x^{\infty} \mathrm{e}^{\beta y}y^{\nu} K_{\nu}(y)\,\mathrm{d}y\bigg|,  \nonumber 
\end{eqnarray}
where $\tilde{h}=h(x)-\widetilde{\mathrm{VG}}^{\nu,1}_{\beta,0}h$.  From inequalities (\ref{vgel11}), (\ref{vgel22}) and (\ref{vgel33}) it follows that the expressions involving modified Bessel functions are bounded for all $x\geq 0$.  Recalling Remark \ref{mark}, it is sufficient to bound these expressions in the region $x \geq0$ and then consider the case of both positive and negative $\beta$, and so we have shown the $f$ and its first derivative are bounded for all $x\in\mathbb{R}$. 

\subsection{Proof of Lemma \ref{lemsec4}}
Taylor expanding $f''(W)$ about $ST_j$ gives
\begin{align*}\mathbb{E}ST^2f''(W)&=\frac{1}{\sqrt{n}}\sum_{j=1}^n\mathbb{E}Y_j\bigg(T_j+\frac{1}{\sqrt{n}}Y_j\bigg)Sf''(W) \\
&=\frac{1}{\sqrt{n}}\sum_{j=1}^n\mathbb{E}Y_j\bigg(T_j+\frac{1}{\sqrt{n}}Y_j\bigg)S\bigg(f''(ST_j) \\
&\quad+\frac{1}{\sqrt{n}}Y_jSf^{(3)}(ST_j) 
+\frac{1}{2n}Y_j^2S^2f^{(4)}(ST_j^{[1]})\bigg) \\
&=R_{1}+N_1+R_{2},
\end{align*}
where
\begin{eqnarray*}R_{1}&=&\frac{1}{n}\sum_{j=1}^n\mathbb{E}Sf''(ST_j), \\
N_1&=&\frac{1}{n}\sum_{j=1}^n\mathbb{E}S^2\bigg(Y_j^2T_j+\frac{1}{\sqrt{n}}Y_j^3\bigg)f^{(3)}(ST_j), \\
|R_{2}|&\leq&\frac{\|f^{(4)}\|}{2n^{3/2}}\sum_{j=1}^n\mathbb{E}S^3\bigg(Y_j^3T_j+\frac{1}{\sqrt{n}}Y_j^4\bigg)\leq\frac{\|f^{(4)}\|}{2\sqrt{n}}\bigg(\mathbb{E}|Y^3|+\frac{\mathbb{E}Y^4}{\sqrt{n}}\bigg)\bigg(3+\frac{\mathbb{E}X^4}{m}\bigg)^{3/4}.
\end{eqnarray*}
Here we used that $\mathbb{E}Y_j=0$ and $\mathbb{E}Y_j^2=1$ to simplify $R_{1}$.  To bound $R_{1}$ we Taylor expand $f''(ST_j)$ about $W$ to obtain
\[|R_{1}|\leq|\mathbb{E}Sf''(W)|+\frac{\|f^{(3)}\|}{\sqrt{n}}.\]
We now consider bound $N_{1}$.  Using independence and that $\mathbb{E}Y_j=0$ and $\mathbb{E}Y_j^2=1$ gives
\[N_{1}=\frac{1}{n}\sum_{j=1}^n\mathbb{E}S^2Tf^{(3)}(W_j)+\frac{\mathbb{E}Y^3}{n^{3/2}}\sum_{j=1}^n\mathbb{E}S^2f^{(3)}(ST_j).
\]
Taylor expanding the $f^{(3)}(ST_j)$ about $W$ allows us to write $N_1$ as
\[N_{1}=\mathbb{E}S^2Tf^{(3)}(W)+R_3,\]
where
\begin{align*}|R_3|&=\bigg|\frac{\mathbb{E}Y^3}{n^{3/2}}\sum_{j=1}^n\mathbb{E}S^2f^{(3)}(ST_j)-\frac{1}{n^{3/2}}\sum_{j=1}^n\mathbb{E}Y_jS^3\bigg(T_j+\frac{1}{\sqrt{n}}Y_j\bigg)f^{(4)}(ST_j^{[1]})\bigg|\\
&\leq\frac{\|f^{(3)}\|\mathbb{E}|Y^3|}{\sqrt{n}}+\frac{\|f^{(4)}\|}{\sqrt{n}}\bigg(1+\frac{1}{\sqrt{n}}\bigg)\bigg(3+\frac{\mathbb{E}X^4}{m}\bigg)^{3/4}.
\end{align*}
Putting this together, we have shown that
\[\mathbb{E}ST^2f''(W)=\mathbb{E}S^2Tf^{(3)}(W)+R_1+R_2+R_3.\]
Rearranging and apply the triangle inequality now gives
\[|\mathbb{E}S^2Tf^{(3)}(W)|\leq |\mathbb{E}ST^2f''(W)|+|R_1|+|R_2|+|R_3|,\]
and summing up the remainder terms completes the proof.

\subsection{Proof of Lemma \ref{appe33}}
We prove that inequality (\ref{2ones}) holds; inequality (\ref{256ones}) then follows by symmetry.  We begin by obtaining a formula for the third order partial derivative of $\psi$ with respect to $s$.  Using a straightforward generalisation of the proof of Lemma 3.2 of Rai\v{c} \cite{raic} it can be shown that
\begin{equation}\label{1ones}\frac{\partial^3\psi}{\partial s^3} =\int_0^{\infty}\!\int_{\mathbb{R}^2}\frac{\mathrm{e}^{-u}}{\sqrt{1-\mathrm{e}^{-2u}}}\frac{\partial^2}{\partial s^2}g(z_s,z_t)\phi'(x)\phi(y)\,\mathrm{d}x\,\mathrm{d}y\,\mathrm{d}u,
\end{equation}
where $z_s=\mathrm{e}^{-u}s+\sqrt{1-\mathrm{e}^{-2u}}x$, $z_t= \mathrm{e}^{-u}t+\sqrt{1-\mathrm{e}^{-2u}}y$, and $\phi(x)=\frac{1}{\sqrt{2\pi}}\mathrm{e}^{-\frac{1}{2}x^2}$ so that $\phi'(x)=-x\phi(x)$.
We now calculate the second order partial derivative of $g$ with respect to $s$.  Since
\[\frac{\partial z_s}{\partial s}=\frac{\partial z_t}{\partial t}=\mathrm{e}^{-u} \qquad \mbox{and} \qquad \frac{\partial z_t}{\partial s}=\frac{\partial z_s}{\partial t}=0,\]
we have that
\[\frac{\partial^2g}{\partial s^2}=\mathrm{e}^{-2s}\{z_s^az_t^{b+2}f^{(4)}(z_sz_t)+2az_s^{a-1}z_t^{b+1}f^{(3)}(z_sz_t)+a(a-1)z_s^{a-2}z_t^bf''(z_sz_t)\}.
\]
We now use the simple inequality that $|p+q|^n\leq 2^{n-1}(|p|^n+|q|^n)$ to obtain the following bound on $z_s$
\[|z_s^n|=|\mathrm{e}^{-u}s+\sqrt{1-\mathrm{e}^{-2u}}x|^n
\leq 2^{n-1}(\mathrm{e}^{-nu}|s|^n+(1-\mathrm{e}^{-2u})^{n/2}|x|^n)
\leq 2^{n-1}(|s|^n+|x|^n)
\]
and a similar inequality holds for $z_t$.  With these inequalities we have the following bound
\begin{align*}\bigg|\frac{\partial^2g}{\partial s^2}\bigg|&\leq \mathrm{e}^{-2u}\{2^{a+b}\|f^{(4)}\|(|s|^a+|x|^a)(|t|^{b+2}+|y|^{b+2}) \\
&\quad+2a2^{a+b-2}\|f^{(3)}\|(|s|^{a-1}+|x|^{a-1})(|t|^{b+1}+|y|^{b+1}) \\
&\quad+a(a-1)2^{a+b-4}\|f''\|(|s|^{a-2}+|x|^{a-2})(|t|^b+|y|^b)\}.
\end{align*}
Applying this bound to equation (\ref{1ones}) gives the following bound on the third order partial derivative of $\psi$ with respect to $s$:
\begin{align}\bigg|\frac{\partial^3\psi}{\partial s^3}\bigg|&\leq\quad\int_0^{\infty}\!\int_{\mathbb{R}^2}\frac{\mathrm{e}^{-u}}{\sqrt{1-\mathrm{e}^{-2u}}}\bigg|\frac{\partial^2g}{\partial s^2}\bigg||x|\phi(x)\phi(y)\,\mathrm{d}x\,\mathrm{d}y\,\mathrm{d}u \nonumber \\
&\leq \int_0^{\infty}\int_{\mathbb{R}^2}\frac{\mathrm{e}^{-3u}}{\sqrt{1-\mathrm{e}^{-2u}}}\{2^{a+b}\|f^{(4)}\|(|s|^a+|x|^a)(|t|^{b+2}+|y|^{b+2}) \nonumber \\
&\quad+a2^{a+b-1}\|f^{(3)}\|(|s|^{a-1}+|x|^{a-1})(|t|^{b+1}+|y|^{b+1}) \nonumber \\
&\quad+a(a-1)2^{a+b-4}\|f''\|(|s|^{a-2}+|x|^{a-2})(|t|^b+|y|^b)\}|x|\phi(x)\phi(y)\,\mathrm{d}x\,\mathrm{d}y\,\mathrm{d}u \nonumber \\
&=\frac{\pi}{4}\int_{\mathbb{R}^2}\{2^{a+b}\|f^{(4)}\|(|s|^a+|x|^a)(|t|^{b+2}+|y|^{b+2}) \nonumber\\
&\quad+a2^{a+b-1}\|f^{(3)}\|(|s|^{a-1}+|x|^{a-1})(|t|^{b+1}+|y|^{b+1}) \nonumber \\
\label{3ones} &\quad+a(a-1)2^{a+b-4}\|f''\|(|s|^{a-2}+|x|^{a-2})(|t|^b+|y|^b)\}|x|\phi(x)\phi(y)\,\mathrm{d}x\,\mathrm{d}y,
\end{align}
where the final equality follows from the formula $\int_0^{\infty}\frac{\mathrm{e}^{-3u}}{\sqrt{1-\mathrm{e}^{-2u}}}\,\mathrm{d}u=\frac{\pi}{4}$ (see Gaunt \cite{gaunt thesis}, p. 19).  We can now obtain the desired bound by using the following formula (see formula 17 of Winkelbauer \cite{winkelbauer}) to evaluate (\ref{3ones}):
\begin{eqnarray*}\int_{-\infty}^{\infty}|x|^k\phi(x)\,\mathrm{d}x&=&\frac{2^{k/2}\Gamma(\frac{k+1}{2})}{\sqrt{\pi}}=(k-1)!!\begin{cases} \sqrt{\frac{2}{\pi}}, & \quad k \mbox{ odd}, \\
1, & \quad k \mbox { even,} \end{cases} \\
&\leq& (k-1)!!,
\end{eqnarray*}
which completes the proof. 

\section{Elementary properties of modified Bessel functions}

Here we list standard properties of modified Bessel functions that are used throughout this paper.  All these formulas can be found in Olver et al$.$ \cite{olver}, except for the inequalities, which are given in Gaunt \cite{gaunt bes}.

\subsection{Definitions}
The \emph{modified Bessel function of the first kind} of order $\nu \in \mathbb{R}$ is defined, for all $x\in\mathbb{R}$, by
\begin{equation}\label{defI}I_{\nu} (x) = \sum_{k=0}^{\infty} \frac{1}{\Gamma(\nu +k+1) k!} \left( \frac{x}{2} \right)^{\nu +2k}.
\end{equation}

The \emph{modified Bessel function of the second kind} of order $\nu \in \mathbb{R}$ can be defined in terms of the modified Bessel function of the first kind as follows
\begin{eqnarray*}K_{\nu} (x) &=& \frac{\pi}{2 \sin (\nu \pi)} (I_{-\nu}(x) - I_{\nu} (x)), \quad \nu \not= \mathbb{Z}, \: x \in\mathbb{R}, \\
K_{\nu} (x) &=& \lim_{\mu\to\nu} K_{\mu} (x) = \lim_{\mu\to\nu} \frac{\pi}{2 \sin (\mu \pi)} (I_{-\mu}(x) - I_{\mu} (x)), \quad \nu \in \mathbb{Z}, \: x \in\mathbb{R}.
\end{eqnarray*}

\subsection{Basic properties}
For $\nu\in\mathbb{R}$, the modified Bessel function of the first kind $I_{\nu}(x)$ and the modified Bessel function of the second kind $K_{\nu}(x)$ are regular functions of $x$.  For $\nu> -1$ and $x>0$ we have $I_{\nu}(x)>0$ and $K_{\nu}(x)>0$.  For all $\nu\in\mathbb{R}$ the modified Bessel function $K_{\nu}(x)$ is complex-valued in the region $x<0$.

\subsection{Asymptotic expansions}
\vspace{-7mm}
\begin{eqnarray}
\label{Ktend0}K_{\nu} (x) &\sim& \begin{cases} 2^{|\nu| -1} \Gamma (|\nu|) x^{-|\nu|}, & \quad x \downarrow 0, \: \nu \not= 0, \\
-\log x, & \quad x \downarrow 0, \: \nu = 0, \end{cases} \\
 \label{roots} I_{\nu} (x) &\sim& \frac{e^x}{\sqrt{2\pi x}}, \quad x \rightarrow \infty, \\
\label{Ktendinfinity} K_{\nu} (x) &\sim& \sqrt{\frac{\pi}{2x}} \mathrm{e}^{-x}, \quad x \rightarrow \infty.
\end{eqnarray}

\subsection{Identities}
\vspace{-7mm}
\begin{eqnarray} \label{wront} I_{\nu}(x)K_{\nu+1}(x)-I_{\nu+1}(x)K_{\nu}(x)&=&\frac{1}{x}, \\
 \label{imre} (-1)^{\nu}K_{\nu}(x)-\pi i I_{\nu}(x)&=&K_{\nu}(-x).
\end{eqnarray}

\subsection{Differentiation}
\vspace{-7mm}
\begin{eqnarray}
\label{cat} K_{\nu}'(x)&=&-\frac{1}{2}(K_{\nu+1}(x)-K_{\nu-1}(x)), \\
\label{diffIii}\frac{\mathrm{d}}{\mathrm{d}x} \left(\frac{I_{\nu} (x)}{x^{\nu}} \right) &=& \frac{I_{\nu +1} (x)}{x^{\nu}}, \\
\label{diffKii} \frac{\mathrm{d}}{\mathrm{d}x} \left(\frac{K_{\nu} (x)}{x^{\nu}} \right) &=& -\frac{K_{\nu +1} (x)}{x^{\nu}}, 
\end{eqnarray}

\subsection{Modified Bessel differential equation}
The modified Bessel differential equation is 
\begin{equation} \label{realfeel} x^2 f''(x) + xf'(x) - (x^2 +\nu^2)f(x) =0.\end{equation}
The general solution is $f(x)=AI_{\nu} (x) +BK_{\nu} (x).$

\subsection{Inequalities}
Let $-1<\beta<1$ and $n=0,1,2,\ldots$, then for $x\geq 0$ we have
\begin{align}\label{vgel11}\bigg|\frac{\mathrm{d}}{\mathrm{d}x}\bigg(\frac{\mathrm{e}^{-\beta x}K_{\nu}(x)}{x^{\nu}}\bigg)\int_0^x\mathrm{e}^{\beta t}t^{\nu}I_{\nu}(x)\,\mathrm{d}t\bigg|&<\frac{2(\beta+1)}{2\nu+1}xK_{\nu+1}(x)I_{\nu}(x)<\infty, \:\: \nu>-\frac{1}{2}, \\ 
\label{vgel22}\left| \frac{\mathrm{d}^n}{\mathrm{d}x^n}\bigg(\frac{\mathrm{e}^{-\beta x}I_{\nu}(x)}{x^{\nu}}\bigg) \int_x^{\infty} \mathrm{e}^{\beta t}t^{\nu} K_{\nu}(t)\,\mathrm{d}t\right| &< \frac{ \sqrt{\pi} \Gamma(\nu +1/2)}{(1-\beta^2)^{\nu+1/2}\Gamma(\nu+1)}, \quad \nu \geq \frac{1}{2},  \\
\label{vgel33}\left| \frac{\mathrm{d}^n}{\mathrm{d}x^n}\bigg(\frac{\mathrm{e}^{-\beta x}I_{\nu}(x)}{x^{\nu}}\bigg) \int_x^{\infty} \mathrm{e}^{\beta t}t^{\nu} K_{\nu}(t)\,\mathrm{d}t\right|&<\frac{(\mathrm{e}+1)2^{2\nu}\Gamma(\nu+1/2)}{1-|\beta|}, \quad |\nu|<\frac{1}{2}. 
\end{align}

\section*{Acknowledgements}

During the course of this research the author was supported by an EPSRC DPhil Studentship and an EPSRC Doctoral Prize.  The author would like to thank Gesine Reinert for the valuable guidance she provided on this project.  The author would also like to thank two anonymous referees for their helpful comments which have lead to a substantial improvement in the presentation of this paper.



\end{document}